\documentclass[draft]{rs}
\usepackage[cp1250]{inputenc}

\title{Ric\-ci sol\-i\-tons}
\titlenote{Expanded version of a plenary lecture at the 4th Forum of Polish 
Mathematicians in Olsztyn, July 1\hskip-.2pt--\hskip.5pt3, 2010.}

\author[A. Derdzinski]{Andrzej Derdzinski}
\institution{The Ohio State University}
\address[Columbus, Ohio, USA]{231 W.\ 18th Ave.\\43210 Columbus,\hskip4ptOhio,\hskip4ptUSA}
\email{andrzej@math.ohio-state.edu}

\begin{document}

\def\hs{\hskip.7pt}
\def\hh{\hskip.4pt}
\def\nh{\hskip-.7pt}
\def\hn{\hskip-.4pt}
\def\nnh{\hskip-1pt}
\def\bz{b\hs}
\def\bx{B}
\def\cn{\mathcal{C}}
\def\dm{\mathcal{D}}
\def\ex{E}
\def\ey{\mathcal{E}}
\def\ef{f}
\def\hatg{{\hat{g\hskip2pt}\hskip-1.3pt}}
\def\ky{\mathcal{K}}
\def\lx{L}
\def\jp{p}
\def\jq{q}
\def\sn{\mathcal{S}}
\def\js{s}
\def\jt{t}
\def\tc{T}
\def\ju{u}
\def\jv{v}
\def\jw{w}
\def\jx{x}
\def\jy{y}
\def\jz{z}
\def\vg{\varGamma}
\def\dv{\delta}
\def\ve{\varepsilon}
\def\zh{\zeta}
\def\lc{\lambda}
\def\pg{\varLambda}
\def\vx{\varXi}
\def\pr{\varPi}
\def\cf{\sigma}
\def\mf{\varSigma}
\def\kp{\tau}
\def\sy{\psi}
\def\dt{\hskip2.8pt\dot{\hskip-2.8pt\fy}}
\def\dd{\hskip2.8pt\ddot{\hskip-2.8pt\fy}}
\def\fy{\phi}
\def\jh{\chi}
\def\hc{\hbox{$I\hskip-3.3ptH$}}
\def\pc{\hbox{$I\hskip-3.3ptP$}}
\def\bbR{\mathrm{I\!R}}
\def\rto{\bbR\hskip-.5pt^2}
\def\rtr{\bbR\hskip-.7pt^3}
\def\bbC{{\mathchoice {\setbox0=\hbox{$\displaystyle\mathrm{C}$}\hbox{\hbox 
to0pt{\kern0.4\wd0\vrule height0.9\ht0\hss}\box0}} 
{\setbox0=\hbox{$\textstyle\mathrm{C}$}\hbox{\hbox 
to0pt{\kern0.4\wd0\vrule height0.9\ht0\hss}\box0}} 
{\setbox0=\hbox{$\scriptstyle\mathrm{C}$}\hbox{\hbox 
to0pt{\kern0.4\wd0\vrule height0.9\ht0\hss}\box0}} 
{\setbox0=\hbox{$\scriptscriptstyle\mathrm{C}$}\hbox{\hbox 
to0pt{\kern0.4\wd0\vrule height0.9\ht0\hss}\box0}}}}
\def\dimc{\dim_{\hh\bbC\hskip-1.2pt}}
\def\bbCP{\bbC\mathrm{P}}
\def\bbZ{\hbox{\sf Z\hskip-4ptZ}}
\def\df{d\hskip-.8pt\ef}
\def\bn{\hs\overline{\nh\nabla\nh}\hs}
\def\bg{\hskip1.8pt\overline{\hskip-1.8ptg\hskip0pt}\hskip0pt} 
\def\bs{\hskip.2pt\overline{\hskip-.2pt\mathrm{s}\hskip-.4pt}\hskip.4pt} 
\def\bd{\hskip.8pt\overline{\hskip-.8pt\Delta\hskip-.8pt}\hskip.8pt} 
\def\br{\hskip.7pt\overline{\hskip-.7pt\mathrm{Ric}\hskip-.7pt}\hskip.7pt}
\def\hyp{\hskip.5pt\vbox
{\hbox{\vrule width2.5ptheight0.5ptdepth0pt}\vskip2pt}\hskip.5pt}
\def\ie{\hbox{$I_{\ve\phantom/}^\chi$\hskip-2.8pt}}

\newtheorem{question}[theorem]{Question}
\newtheorem*{justification}{Uzasadnienie}
\def\NoBlackBoxes{\global\overfullrule 0pt}
\NoBlackBoxes

\section{The Ric\-ci flow}\label{pr}
In 1981 Richard Hamilton \cite{hamilton} initiated the study of the equation
\begin{equation}\label{rfl}
\frac d{dt}\,g(t)\,=\,\,-2\,\mathrm{Ric}\hh_{g(t)}\hskip.7pt.
\end{equation}
Equation (\ref{rfl}) -- the solutions of which are usually said to form the 
{\it Ric\-ci flow\/} -- is a condition imposed on an unknown smooth curve 
$\,t\mapsto g(t)$ of Riemannian metrics on a fixed manifold $\,M$. The 
condition consists in requiring the curve to have, at every $\,t\,$ in its 
domain interval, the derivative with respect to $\,t\,$ equal to $\,-2\,$ 
times the Ric\-ci tensor of the metric $\,g(t)$.

Solutions (trajectories) of the Ric\-ci flow are curves $\,t\mapsto g(t)\,$ 
emanating from a given initial metric $\,g(0)$, and defined on a maximal 
interval $\,[\hskip.7pt0,T)\,$ of the variable $\,t$, where $\,0<T\le\infty$.

In local coordinates $\,x^{\hs j}\nnh$, $\,j=1,\dots,\dim M$, (\ref{rfl}) 
constitutes a system of nonlinear partial differential equations of parabolic 
type, imposed on the componente $\,g_{jk}=g(e_j,e_k)\,$ of the metrics 
$\,g=g(t)\,$ belonging to our unknown curve. The symbol $\,e_j$ denotes 
here the $\,j$th coordinate vector field (so that the directional derivative 
in the direction of $\,e_j$ coincides with the partial derivative 
$\,\partial_j$ with respect to the $\,j$th coordinate). The functions 
$\,g_{jk}$ depend on $\,t\,$ and on the variables $\,x^{\hs j}\nnh$. The 
coordinate version of condition (\ref{rfl}) is rather complicated:
\[
\frac{\partial g_{jk}}{\partial t}\,=\,-2R_{jk}^{\phantom i}\hh,\hskip14pt
\mathrm{where}\hskip9ptR_{jk}^{\phantom i}
=\partial_p^{\phantom i}\varGamma_{\!jk}^p
-\partial_j^{\phantom i}\varGamma_{\!pk}^p
+\varGamma_{\!qp}^p\varGamma_{\!jk}^p-\varGamma_{\!jq}^p\varGamma_{\!pk}^q,
\]
$\varGamma_{\!jk}^p$ being the Christoffel symbols of the metric $\,g=g(t)$, 
given by $\,2\hs\varGamma_{\!jk}^p=g^{pq}(\partial_jg_{kq}+\partial_kg_{jq}
-\partial_qg_{jk})$, with $\,g^{jk}$ standing for the contravariant 
components of the metric (meaning that, at every point of the coordinate 
domain, the matrix $\,[g^{jk}]\,$ is the inverse of the matrix $\,[g_{jk}]$). 
In the above expressions for $\,R_{jk}^{\phantom i}$ and 
$\,2\hs\varGamma_{\!jk}^p$ we have used the Ein\-stein summing convention, so 
that repeated indices are summed over.

\section{The Ric\-ci flow in the proof of Poin\-car\'e's conjecture}\label{rp}
In \cite{hamilton} Hamilton proved the existence and uniqueness of a maximal 
Ric\-ci flow trajectory with any given initial metric $\,g(0)$, on every 
compact manifold. He also attempted to use this fact for proving the 
three-di\-men\-sion\-al Poin\-car\'e conjecture.

His proposed outline of such a proof (known as ``the Ha\-mil\-to\-n 
programme'') consisted of a specific series of steps. The last of those steps 
were only carried out in 2002 by Grigori Perelman 
\cites{perelman,perelman-r,perelman-f}.

At the same time Perelman also proved the much more general {\it Thur\-ston 
geometrization conjecture for three-di\-men\-sion\-al manifolds}.

A crucial part of Perelman's argument was provided by surgeries, needed when 
the Ric\-ci flow runs into a singularity in finite time ($T<\infty$). After 
the surgery the Ric\-ci flow is used again, in a topologically simpler 
situation.

\section{Ric\-ci sol\-i\-tons -- ``fixed points'' of the Ric\-ci 
flow}\label{sr}
A {\it Ric\-ci sol\-i\-ton\/} is a Riemannian metric $\,g=g(0)\,$ on a 
manifold $\,M$, evolving under the Ric\-ci flow in an {\it inessential\/} 
manner, in the sense that all the stages $\,g(t)\,$ coincide with $\,g(0)\,$ 
up to dif\-feo\-mor\-phisms and multiplications by positive constants 
(``rescalings''). In other words, such a metric represents a fixed point of 
the Ric\-ci flow in the quotient of the space of metrics on $\,M\,$ under 
the equivalence relation just described.

It is clear what the above definition means in the case of compact manifolds, 
due to the existence and uniqueness of a maximal Ric\-ci flow trajectory 
with any given initial metric. Without the compactness assumption, by a 
Ric\-ci sol\-i\-ton one means a Riemannian metric $\,g\,$ on a manifold 
$\,M$, for which equation (\ref{rfl}) has a solution satisfying the initial 
condition $\,g(0)=g\,$ and constituting an inessential evolution of the metric.

In \S\ref{rs} we will discuss a different characterization of Ric\-ci 
sol\-i\-tons, having the form of a differential equation.

\section{Blow-up limits}\label{gt}
Complete noncompact Ric\-ci sol\-i\-tons often arise as by-prod\-ucts of the 
Ric\-ci flow on compact manifolds -- blow-up limits (or, rescaling limits) of 
the metrics $\,g(t)\,$ restricted to suitable open sets, when the variable 
$\,t\in[\hs0,T)$ tends to $\,T\,$ (defined in \S\ref{pr}), while 
$\,T\,$ is finite.

The following result of Perelman \cite{perelman} (known as the 
``no\hskip.7pt-breathers theorem'') states that on compact manifolds 
one can eqivualently characterize Ric\-ci sol\-i\-tons by a condition 
seemingly much weaker than their original definition: 
\begin{theorem}
If in a trajectory\/ $\,t\mapsto g(t)\,$ of the Ric\-ci flow on a compact 
manifold there exist two distinct values of\/ $\,t\,$ such that the 
corresponding stages of the flow coincide up to a dif\-feo\-mor\-phism and 
rescaling, then the trajectory is a Ric\-ci sol\-i\-ton.
\end{theorem}
In dimension $\,3\,$ this was first proved by Thomas Ivey \cite{ivey}.

\section{The Ric\-ci-sol\-i\-ton equation}\label{rs}
One easily verifies that a metric $\,g\,$ on a manifold $\,M\,$ is a 
Ric\-ci sol\-i\-ton if and only if some vector field $\,w\,$ on $\,M\,$ 
satisfies the {\it Ric\-ci-sol\-i\-ton equation}
\begin{equation}\label{sol}
\pounds\nnh_w\phantom{^i}\hskip-3ptg\hh\,+\,\hh\mathrm{Ric}\hh\,
=\,\hskip.7pt\lambda\hskip.7pt g\hs,\hskip14pt\mathrm{where}\hskip7pt
\lambda\hskip7pt\text{\rm is\ a\ constant,}
\end{equation}
with $\,\mathrm{Ric}\,$ denoting the Ric\-ci tensor of $\,g$, and 
$\,\pounds\nnh_w\phantom{^i}\hskip-3ptg\,$ the Lie derivative of $\,g\,$ along 
$\,w$.

The term {\it Ric\-ci sol\-i\-ton\/} is also used for a Riemannian manifold 
$\,(M,g)$ satisfying (\ref{sol}) with some $\,w$, as well as for 
a Ric\-ci flow trajectory $\,t\mapsto g(t)$ in which the initial stage 
$\,g(0)\,$ (or, equivalently, every stage $\,g(t)$) has the property 
(\ref{sol}); the field $\,w\,$ may here depend on $\,t$.

The objects $\,w\,$ and $\,\lambda\,$ appearing in (\ref{sol}) will be called 
the {\it sol\-i\-ton vector field\/} and {\it sol\-i\-ton constant}.

The coordinate version of condition (\ref{sol}) reads
\[
w_{j\hh,\hs k}^{\phantom i}\hs+\hs w_{k\hh,\hh j}^{\phantom i}\hs
+\hs R_{jk}^{\phantom i}\,=\,\lambda\hskip.7pt g_{jk}^{\phantom i}\hh.
\]
Instead of $\,w_{j\hh,\hs k}^{\phantom i}$ one also writes 
$\,\nabla_{\!k}^{\phantom i}w_j^{\phantom i}$. One may express (\ref{sol}) 
directly in terms of the components $\,g_{jk}$ of $\,g$, and $\,w^{\hs j}$ of 
$\,w$, by replacing $\,R_{jk}^{\phantom i}$ with the formula in \S\ref{pr}, 
and the sum $\,w_{j\hh,\hs k}^{\phantom i}+\hs w_{k\hh,\hh j}^{\phantom i}$ 
with $\,\partial_k^{\phantom i}w_j^{\phantom i}
+\partial_j^{\phantom i}w_k^{\phantom i}
-2\hs\varGamma_{\!jk}^pw_p^{\phantom i}$, for 
$\,w_j^{\phantom i}=g_{jk}^{\phantom i}w^{\hs k}\nnh$.

\section{The topics discussed below}\label{kd}
Ric\-ci sol\-i\-tons are of obvious interest, due to their close relation 
with the Ric\-ci flow (\S\S\ref{sr}--\/\ref{gt}). Their characterization as 
Riemannian metrics satisfying a specific system of partial differential 
equations (\S\ref{rs}) suggests applying to their study the methods of 
geometric analysis that were used previously in similar situations.

The remaining part of this text deals with a few selected cases in which the 
approach just mentioned leads to better understanding of Ric\-ci sol\-i\-tons 
on {\it compact\/} manifolds. Examples of such Ric\-ci sol\-i\-tons are 
provided by Ein\-stein metrics (\S\ref{me}), K\"ah\-ler-Ric\-ci sol\-i\-tons 
(\S\ref{sk}), and some Riemannian products (\S\ref{po}). In contrast with the 
lowest dimensions, $\,n=2\,$ and $\,n=3$, where compact Ric\-ci sol\-i\-tons 
have long been classified (\S\ref{me}), the case $\,n\ge4\,$ is largely a {\it 
terra incognita}, as illustrated by the open problem described in \S\ref{po}.

\S\S\ref{gs}--\ref{dt} are devoted to stating and proving a result of Perelman 
which -- despite its rather esoteric nature -- is a substantial step toward 
understanding compact Ric\-ci sol\-i\-tons. In \S\ref{km} we in turn discuss 
theorems, proved in the years 1957--2004 by eight mathematicians, and together 
showing that K\"ah\-ler-Ric\-ci sol\-i\-tons form a natural class of {\it 
canonical metrics\/} on compact complex surfaces with positive or negative 
first Chern class.

The final part of this article (\S\S\ref{rh}--\ref{pt}) presents two classes 
of examples of Ric\-ci sol\-i\-tons on compact manifolds. They are, namely, 
Page's and B\'erard Bergery's Ein\-stein metrics \cite{page,berard-bergery}, 
and the Koi\-\hbox{so\hskip.7pt-}\hskip0ptCao K\"ah\-ler metrics 
\cite{koiso,cao-eg}. Their description is rather technical and to go through 
it one has to ``roll up the sleeves'' -- in contrast, for instance, to the 
case of homogeneous Ein\-stein manifolds with an irreducible isotropy 
representation, where the discussion and justification is very brief 
(\S\ref{me}). The constructions in \S\S\ref{rh}--\ref{pt} use con\-for\-mal 
changes of K\"ah\-ler metrics, that is, their multiplication by suitable 
positive functions. Conditions sufficient for such a change to yield a Ric\-ci 
sol\-i\-ton, introduced in \S\ref{kk}, constitute a system of sec\-ond-or\-der 
ordinary differential equations with boundary conditions. Known solutions of 
this system form three families, two of which correspond 
(see \S\S\ref{pb},\hs\ref{pk}) to the two classes of examples mentioned above, 
while the third one yields nothing new -- the Ric\-ci sol\-i\-tons arising in 
it are isometric to the Koi\-\hbox{so\hskip.7pt-}\hskip0ptCao metrics, as 
shown by Gideon Maschler \cite{maschler}. A proof of Maschler's result is 
given in \S\ref{pt}.

\section{Ein\-stein metrics}\label{me}
The most obvious class of Ric\-ci sol\-i\-tons on compact manifolds is 
provided by {\it Ein\-stein metrics}. They are defined to be Riemannian 
metrics $\,g\,$ satisfying (\ref{sol}) with $\,w=0$, that is, the {\it 
Ein\-stein condition}
\begin{equation}\label{ein}
\mathrm{Ric}\hh\,=\,\hskip.7pt\lambda\hskip.7pt g\hskip14pt\text{\rm for\ 
some\ constant}\hskip7pt\lambda\hh.
\end{equation}
The sol\-i\-ton constant $\,\lambda\,$ is then called the {\it Ein\-stein 
constant}.

In dimensions $\,n<4$, every compact Ric\-ci sol\-i\-ton is an Ein\-stein 
metric; this was proved by Hamilton \cite{hamilton-tr} for $\,n=2\,$ and by 
Ivey \cite{ivey} for $\,n=3$. For purely algebraic reasons, when $\,n<4$, 
equation (\ref{ein}) implies that the metric $\,g\,$ has constant 
sectional curvature. Locally, up to isometries and rescalings, such 
low-di\-men\-sion\-al manifolds are thus standard spheres, Euclidean spaces, 
or hyperbolic spaces.

By an {\it Ein\-stein manifold\/} one means a Riemannian manifold $\,(M,g)$ 
such that $\,g\,$ is an Ein\-stein metric.

The Ein\-stein property of the constant-curvature metrics mentioned above 
(spherical, Euclidean, hyperbolic) also follows for much more general reasons. 
Namely, {\it every homogeneous Riemannian manifold\/ $\,(M,g)$ with an 
irreducible isotropy representation is an Ein\-stein manifold}. The 
homogeneity assumption means that the isometry group of $\,g\,$ acts on 
$\,M\,$ transitively, while {\it irreducibility\/} refers here to the 
group $\,H_x$ of those isometries keeping a given point $\,x\in M\,$ fixed; 
$\,H_x$ acts, infinitesimally, on the tangent space $\,T_xM$. The Ein\-stein 
condition is here a trivial consequence of Schur's lemma and the fact that 
the Ric\-ci tensor -- being a natural invariant of the metric -- is 
preserved by all isometries.

Besides the spherical, Euclidean and hyperbolic  metrics, the above 
theorem also applies, for instance, to the canonical (Fu\-bi\-ni-Stu\-dy) 
metrics on complex projective spaces $\,\bbCP^m\nnh$, showing that, 
consequently, they are Ein\-stein metrics.

\section{K\"ah\-ler-Ric\-ci sol\-i\-tons}\label{sk}
First examples of non\hskip.7pt-Ein\-stein compact Ric\-ci sol\-i\-tons, 
representing all even dimensions $\,n\ge4$, were constructed in the early 
1990s by Norihito Koiso \cite{koiso} and (independently) Huai-Dong Cao 
\cite{cao-eg}. All their examples, as well as generalizations of those 
examples found by other authors 
\cites{feldman-ilmanen-knopf,dancer-wang,li-c}, are {\it K\"ah\-ler-Ric\-ci 
sol\-i\-tons}, in the sense of being, simultaneously, Ric\-ci sol\-i\-tons and 
K\"ah\-ler metrics.

A special case of K\"ah\-ler-Ric\-ci sol\-i\-tons is provided by {\it 
K\"ah\-ler-Ein\-stein metrics}, that is, K\"ah\-le\-r metrics which are also 
Ein\-stein metrics.

Recall that one of the possible (mutually equivalent) definitions of a 
K\"ah\-ler metric can be phrased as follows. It is a metric $\,g\,$ on a 
manifold $\,M\,$ such that some fixed linear automorphism $\,J\,$ of the 
tangent bundle $\,T\nnh M\,$ (that is, some smooth family $\,x\mapsto J_x$ of 
linear automorphisms $\,J_x:T_xM\to T_xM$) satisfies the conditions 
$\,J(Jv)=-\hh v$, $\,g(Jv,Jw)=g(v,w)$ and 
$\,\nabla_{\!v}(Jw)=J(\nabla_{\!v}w)\,$ for all smooth vector fields 
$\,v\,$ and $\,w$, where $\,\nabla\,$ denotes the Le\-vi-Ci\-vi\-ta 
connection of the metric $\,g$.

By a {\it K\"ah\-ler metric on a complex manifold\/} $\,M\,$ we in turn mean a 
K\"ah\-ler metric, in the above sense, on $\,M\,$ (treated as real manifold), 
for which the automorphism $\,J:T\nnh M\to T\nnh M\,$ with the required 
properties is the operator of multiplication by $\,i\,$ in the tangent spaces 
(naturally constituting complex vector spaces).

The details of the Koi\-\hbox{so\hskip.7pt-}\hskip0ptCao construction will be 
discussed in \S\ref{pk}. For now it should be mentioned that the 
Koi\-\hbox{so\hs-}\hskip0ptCao examples are K\"ah\-ler metrics on compact 
complex manifolds $\,M^m_k$ which, for integers $\,m,k$ with $\,m>k>0$, 
are defined as follows:
\vskip-7pt
\[
\begin{array}{l}
\,\,\hs M^m_k\text{\it\ is\ the\ total\ space\ of\ the\ holomorphic\ 
}\,\bbCP^1\nnh\nh\text{\it\ bundle\ over\ the}\\
\text{\it\ projective\ space\ 
}\,\nh\bbCP^{m-1}\nnh,\text{\it\ arising\ as\ the\ projective\ 
compactifica}\hs\hyp\hskip15pt(*)\\
\text{\it\ tion\ of\ the\ }\hs k\hyp\text{\it th\ tensor\ power\ of\ the\ 
tautological\ line\ bundle.}\\
\end{array}
\]
\vskip4pt
\noindent 
Thus, $\,m\,$ is the complex dimension of $\,M^m_k$. Note that $\,M^m_1$ can 
also be obtained by blowing up a point in $\,\bbCP^m\nnh$.

The same complex manifolds $\,M^m_k$ carry other, no less interesting Ric\-ci 
sol\-i\-tons, which are con\-for\-mally-K\"ah\-ler (though non-K\"ah\-ler) 
Ein\-stein metrics. They were constructed by Don N. Page \cite{page} for 
$\,m=2$ and by Lionel B\'erard Bergery \cite{berard-bergery} in dimensions 
$\,m\ge3$.

By a {\it con\-for\-mal\-ly-K\"ah\-ler metric\/} on a manifold $\,M\,$ we mean 
a Riemannian metric $\,g\,$ on $\,M\,$ admitting a positive function 
$\,\mu:M\to\bbR$ such that the product $\,\mu g\,$ is a K\"ah\-le\-r metric.

For details of Page's and B\'erard Bergery's examples, see \S\ref{pb}.

In complex dimensions $\,m\ge3$, the constructions presented here can be 
directly generalized -- as pointed out by their 
authors themselves \cite{koiso,cao-eg,berard-bergery} -- to a class of compact 
complex manifolds slightly larger than the family $\,M^m_k$ described above. 
Our discussion focuses on the manifolds $\,M^m_k$ for a practical reason: the 
definition of the larger class is rather cumbersome. That definition will, 
however, be introduced in the proof of Lemma~\ref{rhxpl}, which constitutes 
the initial step of the construction, thus allowing the reader to carry out 
the generalization just mentioned.

\section{An open problem}\label{po}
The Riemannian product of two Ric\-ci sol\-i\-tons having the same 
sol\-i\-ton constant $\,\lambda\,$ is again a Ric\-ci sol\-i\-ton. Using this 
fact, the Koi\-\hbox{so\hskip.7pt-}\hskip0ptCao examples (\S\ref{sk}), and 
Ein\-stein metrics (\S\ref{me}), one easily constructs non-Ein\-stein, 
non-K\"ah\-ler compact Ric\-ci sol\-i\-tons of any dimension $\,n\ge7$.

Gang Tian \cite{tian-pc}, in various lectures, and Huai-Dong Cao, in the paper 
\cite{cao-rp}, raised the following
\begin{question}\label{isthr}Does there exist a compact Ric\-ci sol\-i\-ton 
which is neither Ein\-stein nor locally K\"ah\-ler, and is not locally 
decomposable into a Riemannian product of lower-di\-men\-sion\-al manifolds?
\end{question}
If such an example exists, it must have a positive sol\-i\-ton constant; a 
finite fundamental group; and a scalar curvature which is both nonconstant and 
positive. The reason is that compact Ric\-ci sol\-i\-tons which fail to 
satisfy one of the conditions just listed are necessarily Ein\-stein metrics. 
These four facts were proved, respectively, by Jean-Pierre Bourguignon 
\cite{bourguignon}, back in 1974; Xue-Mei Li \cite{li-xm} in 1993 (and, 
independently, Manuel Fer\-n\'an\-dez-L\'o\-pez and Eduardo 
Gar\-c\'\i a-R\'\i o \cite{fernandez-lopez-garcia-rio} in 2004, as well as 
Zhenlei Zhang \cite{zhang} in 2007); Daniel H.\ Friedan \cite{friedan} in 
1985; and Ivey \cite{ivey} in 1993. The results of Hamilton 
\cite{hamilton-tr} and Ivey \cite{ivey} mentioned in \S\ref{me} also show that 
such an example would be of dimension $\,n\ge4$.

\section{Some notations and identities}\label{to}
To simplify our discussion, we introduce some symbols. Given a Riemannian 
manifold $\,(M,g)$, we let $\,\mathcal{F}M,\mathcal{X}M,\varOmega M\,$ and 
$\,SM$ denote the vector spaces of all smooth functions $\,M\to\bbR$, smooth 
vector fields on $\,M$, smooth $\,1$-forms on $\,M\,$ (that is, sections of 
the cotangent bundle) and, respectively, smooth symmetric $\,2$-ten\-sor 
fields on $\,M$. Examples of the latter are the metric $\,g$, its Ric\-ci 
tensor $\,\mathrm{Ric}$, and the Hess\-i\-an $\,\nabla\df\,$ of any function 
$\,f\in\mathcal{F}M$, that is, the covariant derivative of its differential  
$\,\df$. Besides the gradient $\,\nabla:\mathcal{F}M\to\mathcal{X}M\,$ and the 
differential $\,d:\mathcal{F}M\to\varOmega M$, interesting linear operators 
between pairs of these spaces also include the $\,g${\it-trace\/} 
$\,\mathrm{tr}_g:SM\to\mathcal{F}M$, as well as the $\,g${\it-divergence\/} 
$\,\dv:SM\to\varOmega M$, the {\it interior product\/} 
$\,\imath_v:SM\to\varOmega M$ by any $\,v\in\mathcal{X}M$, and 
the $\,g${\it-La\-pla\-cian\/} $\,\Delta:\mathcal{F}M\to\mathcal{F}M$, 
characterized by $\,\mathrm{tr}_gb=\mathrm{tr}\hskip2ptB$, if $\,b\in SM\,$ 
and $\,B\,$ is the linear en\-do\-mor\-phism of the tangent bundle 
$\,T\nnh M\,$ such that
\begin{equation}\label{gbv}
g(Bv,w)\,=\,b(v,w)\hskip14pt\mathrm{for}\hskip8ptv,w\in\mathcal{X}M\hh,
\end{equation}
and $\,(\dv b)_j^{\phantom i}
=g^{\hs pq}\nabla_{\!p}^{\phantom i}b_{pj}^{\phantom i}$, 
$\,\imath_vb=b(v,\,\cdot\,)\,$ and $\,\Delta f=\hh\mathrm{tr}_g\nabla\df$.

The $\,g${\it-trace\-less part\/} of a symmetric $\,2$-ten\-sor field 
$\,b\in SM\,$ is, by definition, the tensor field   
$\,\{b\}_0^{\phantom j}\in SM\,$ given by 
\begin{equation}\label{trz}
\{b\}_0^{\phantom j}\hs\,=\,\,\,b\hs\,\,-\,\,\hs(\mathrm{tr}_gb)\hs g/n\hs,
\hskip12pt\mathrm{where}\hskip8ptn\,=\,\dim M.
\end{equation}
The function $\,\mathrm{s}=\mathrm{tr}_g\mathrm{Ric}\,$ is called the {\it 
scalar curvature\/} of the metric $\,g$. The well-known identities (see, 
e.g.,\ \cite[formulae (2.4) and (2.9)]{derdzinski-maschler-lc})
\begin{equation}\label{toz}
\mathrm{a)}\hskip8pt2\hs\dv\hs\mathrm{Ric}=d\hh\mathrm{s}\hs,
\hskip14pt\mathrm{b)}\hskip8pt\dv b=dY\nh+\imath_v\mathrm{Ric}\hs,
\hskip14pt\mathrm{c)}\hskip8pt2\hs\imath_vb=d\hh Q
\end{equation}
hold for any function $\,f\in\mathcal{F}M$, its gradient $\,v=\nabla\!f\nnh$, 
Hess\-i\-an $\,b=\nabla\df\nnh$, its La\-pla\-cian $\,Y\nnh=\Delta f\nnh$, and 
$\,Q=g(v,v)$. We will also use the {\it nonlinear\/} differential operator 
$\,\mathcal{R}:\mathcal{F}M\to\mathcal{F}M$, given, in any Riemannian manifold 
$\,(M,g)$, by
\begin{equation}\label{rfe}
\mathcal{R}f\,=\,\hs\Delta f\,-\,\hs|\nabla\!f|^2\nh/\hh2\hs.
\end{equation}
One says that a symmetric $\,2$-ten\-sor field $\,b\in SM\,$ on a K\"ah\-ler 
manifold $\,(M,g)\,$ is {\it Her\-mit\-i\-an\/} if the en\-do\-mor\-phism 
$\,B:T\nnh M\to T\nnh M\,$ characterized by (\ref{gbv}) above is 
$\,\bbC$-lin\-e\-ar (that is, commutes with $\,J$). This is equivalent to 
skew-sym\-me\-try of the $\,2$-ten\-sor field $\,b(J\,\cdot\,,\,\cdot\,)$. For 
every $\,x\in M$ the operator $\,B_x:T_xM\to T_xM\,$ must then be 
di\-ag\-o\-nal\-izable and its eigen\-val\-ues have even multiplicities over 
$\,\bbR$. 
In particular,
\begin{equation}\label{grh}
g\hskip6pt\mathrm{and}\hskip7pt\mathrm{Ric}\hskip7pt\text{\rm are\ always\ 
Her\-mit\-i\-an.}
\end{equation}
Below -- and in the sequel -- the symbols $\,\jt\,$ and $\,\tc\,$ will be used 
in a way which has nothing to do with their meanings in \S\S\ref{pr}--\ref{rs}.

For functions $\,\jt,\jh:M\to\bbR\,$ on a manifold $\,M$, we will call 
$\,\jh\,$ a {\it  smooth function of\/} $\,\jt\,$ if $\,\jt\,$ is nonconstant 
and $\,\jh=G\circ\jt$, where $\,G$ is a smooth function on the interval 
$\,\varLambda=\jt(M)$, that is, on the range of $\,\jt$. Writing 
$\,(\,\,)\dot{\,}=\,d/d\jt$, we will form the first and second derivatives 
$\,\dot\jh=\dot G\circ\jt,\,\ddot\jh=\ddot G\circ\jt$, treating them 
simultaneously as functions $\,M\to\bbR$ and as functions of the variable 
$\,\jt\in\varLambda$. With a fixed Riemannian metric $\,g\,$ on $\,M$, for 
$\,\jh,\cf$, which are smooth functions of a  nonconstant function 
$\,\jt:M\to\bbR$, setting $\,\fy=e^\jt g(\nabla\jt,\nh\nabla\jt)/2$, we have 
the obvious equalities $\,\nabla\nnh\jh=\dot\jh\nabla\jt\,$ and
\begin{equation}\label{gns}
\mathrm{i)}\hskip8ptg(\nabla\hskip-1.5pt\cf,\nabla\nnh\jh)
=2\hh e^{-\jt}\dot\cf\dot\jh\hh\fy\hs,\hskip24pt
\mathrm{ii)}\hskip8ptd\jh=\dot\jh\,d\jt\hh.
\end{equation}
In addition, $\,\nabla d\jh=\dot\jh\nabla d\jt+\ddot\jh\,d\jt\otimes d\jt$. so 
that $\,\nabla d\hh e^\jt=e^\jt(\nabla d\jt+\hs d\jt\otimes d\jt)$ which, 
for $\,\kp=e^\jt\nnh$, gives 
$\,\nabla d\jt=e^{-\jt}\nabla d\kp-\hs d\jt\otimes d\jt$, and
\begin{equation}\label{ndh}
\nabla d\jh=\dot\jh e^{-\jt}\nabla d\kp
+(\ddot\jh-\dot\jh)\hs d\jt\otimes d\jt\hs.
\end{equation}

\section{Gradient Ric\-ci sol\-i\-tons}\label{gs}
One says that a given Ric\-ci sol\-i\-ton $\,(M,g)\,$ is {\it of the gradient 
type\/} when the sol\-i\-ton vector field $\,w\,$ with (\ref{sol}) (for some 
constant $\,\lambda$), may be chosen so as to be the gradient of a function. A 
Riemannian manifold $\,(M,g)\,$ is a gradient (type) Ric\-ci sol\-i\-ton if 
and only if there exists a {\it sol\-i\-ton function\/} 
$\,f:M\to\mathrm{I\!R}\,$ satisfying the {\it gradient sol\-i\-ton equation}
\begin{equation}\label{ndf}
\nabla\df\,+\,\hh\mathrm{Ric}\hh\,
=\,\hskip.7pt\lambda\hskip.7pt g\hs,\hskip14pt\mathrm{where}\hskip7pt
\lambda\hskip7pt\text{\rm is\ a\ constant.}
\end{equation}
The symbol $\,\nabla\,$ denotes here the Le\-vi-Ci\-vi\-ta connection of the 
metric $\,g\,$ (but will also be used for the $\,g$-grad\-i\-ent operator), 
while $\,\nabla\df\,$ is, as in \S\ref{to}, the Hess\-i\-an of $\,f$.

That (\ref{sol}) becomes (\ref{ndf}) if $\,2w=\nabla\!f\,$ (that is, when 
$\,2w\,$ is the gradient of a function $\,f$) follows from the identity 
$\,\pounds\nnh_vg=2\nabla\df$, valid for any function $\,f\,$ and its gradient 
$\,v=\nabla\!f$.

The gradient sol\-i\-ton equation (\ref{ndf}) has the following interesting 
consequences, first noted by Ha\-mil\-ton, cf.\ also \cite[p.\ 201]{chow}:
\begin{lemma}\label{stale}Condition\/ {\rm(\ref{ndf})} for a function\/ 
$\,f\,$ on a Riemannian manifold\/ $\,(M,g)\,$ implies constancy of three 
functions\/{\rm:} $\,\Delta\ef+\hs\mathrm{s}$, 
$\,\Delta\ef-g(\nabla\!\ef,\nabla\!\ef)+2\lc\ef$, and\/ 
$\,\mathcal{R}f+\hh\lambda f+\hs\mathrm{s}/2$, where\/ $\,\mathrm{s}\,$ is the 
scalar curvature, and\/ $\,\mathcal{R}f\nh=\Delta f\nh-|\nabla\!f|^2\nh/\hh2$.
\end{lemma}
\begin{proof}[Proof]Applying to both sides of (\ref{ndf}) the operators  
$\,\mathrm{tr}_g$, $\,d\circ\mathrm{tr}_g-2\hs\dv+2\hs\imath_v$ and 
$\,\dv-\imath_v$, we get our three claims as trivial consequences of 
(\ref{toz}).
\end{proof}
Constancy of the last two functions in the lemma is due to vanishing of their 
differentials, since all manifolds are here -- \hbox{by definition -- 
connected.}

Perelman \cite{perelman} proved the gradient property of compact Ric\-ci 
sol\-i\-tons:
\begin{theorem}\label{kzs}
Every compact Ric\-ci sol\-i\-ton is of the gradient type.
\end{theorem}
The proof of this theorem in \S\ref{dt} is based on solvability of certain 
qua\-si-lin\-e\-ar elliptic equations, established by Oscar S.\ Rothaus 
\cite{rothaus} in 1981. Rothaus's result may be stated as follows:
\begin{theorem}\label{rot}For a compact Riemannian manifold\/ $\,(M,g)\,$ of 
dimension\/ $\,n\ge3$, the operator\/ $\,\mathcal{R}\,$ given by\/ 
{\rm(\ref{rfe})}, and any positive real number\/ $\,\lambda$, the assignment\/ 
$\,f\,\mapsto\,\mathcal{R}f+\hh\lambda f\,$ is a surjective mapping of the 
space of smooth functions\/ $\,M\to\bbR\,$ onto itself.
\end{theorem}
The proof of Theorem~\ref{rot}, outlined in \S\ref{sd}, uses the fact 
presented in the next section.

\section{Logarithmic So\-bo\-lev inequalities}\label{ln}
The above term refers to a type of estimates, first studied in the late 1960s 
\cite{federbush,gross}. Rothaus's version \cite{rothaus} may be phrased as 
part (c) of the following lemma, in which, for $\,p\in[\hs1,\infty)\,$ and a 
compact Riemannian manifold $\,(M,g)\,$ of dimension $\,n$, we denote by 
by $\,\|\hskip2.3pt\|_p^{\phantom i}$ and 
$\,\|\hskip2.3pt\|_{p,1}^{\phantom i}$ the $\,L^p$ norm and its associated 
first So\-bo\-lev norm: 
$\,\|\varphi\|_{p,1}^p=\int_M[|\nabla\varphi|^p+\varphi\hh^p]\,\mathrm{d}g$ 
whenever $\,\varphi\in\mathcal{F}M$, where $\,\mathcal{F}M\,$ is the space of 
all smooth functions $\,M\to\bbR$, and $\,\mathrm{d}g\,$ the volume element of 
the metric $\,g$. In addition, we use the symbol $\,L^p_1M\,$ for the 
So\-bo\-lev space obtained as the completion of $\,\mathcal{F}M\,$ in the norm 
$\,\|\hskip2.3pt\|_{p,1}^{\phantom i}$, and identified in an obvious manner 
with a sub\-space of $\,L^p\nh M$. The classical {\it So\-bo\-lev 
inequality\/} and the resulting inclusion state that, whenever $\,p\in(1,n)\,$ 
and $\,\varphi\in\mathcal{F}M$,
\begin{equation}\label{nso}
\|\varphi\|_r^{\phantom i}\le C\|\varphi\|_{p,1}^{\phantom i}\hskip5pt
\text{\rm and}\hskip6ptL^p_1M\subset L^r\hskip-2ptM\hh,\hskip6pt
\text{\rm if}\hskip6pt1\le r\le np/(n-p)\hh,
\end{equation}
for a constant $\,C\,$ depending only on $\,(M,g),p\,$ and $\,r$.
\begin{lemma}\label{logns}For a compact Riemannian manifold\/ $\,(M,g)\,$ of 
dimension\/ $\,n\ge3$, a constant\/ $\,\ve\in\bbR$, \hbox{and a smooth 
function\/ $\hs\chi\nnh:\nnh M\nh\to\bbR$, let}
\begin{equation}\label{fnc}
\ie(\varphi)\,=\,\displaystyle{\int_M\,(\ve|\nabla\varphi|^2
-\varphi\hh^2\log\hh|\varphi|+\chi\hh\varphi^2)\,\mathrm{d}g}\hs,
\end{equation}
where\/ $\,\varphi\hs\log\hh|\varphi|\,$ by definition equals\/ $\,0\,$ on the 
zero set of\/ $\,\varphi$. Then the functional\/ $\,\ie:L^2_1M\to\bbR$, 
defined by\/ {\rm(\ref{fnc})},
  \begin{equivalence}
  \item is well defined, since\/ 
$\,\varphi\hh^2\log\hh|\varphi|\in L^1\hskip-1.3ptM\,$ whenever\/ 
$\,\varphi\in L^2_1M$,
  \item is continuous with respect to the So\-bo\-lev norm\/ 
$\,\|\hskip2.3pt\|_{2,1}^{\phantom i}$,
  \item assumes a minimum value on the set\/ 
$\,\varSigma=\{\varphi\in L^2_1M:\|\varphi\|_2^{\phantom i}=1\}$, provided 
that\/ $\,\ve>0$.
  \end{equivalence}
Every\/ $\,\varphi\in\varSigma\,$ realizing the minimum\/ $\,\kappa\,$ of the 
functional\/ $\,\ie$ on\/ $\,\varSigma$, when\/ $\,\ve>0$, is in addition a 
distributional solution of the equation
\begin{equation}\label{dys}
\ve\hh\Delta\varphi\,+\,\varphi\hs\log\hh|\varphi|\,
+\,(\kappa-\chi)\hs\varphi\,=\,0\hh.
\end{equation}
\end{lemma}
To prove both Lemma~\ref{logns} (in \S\ref{dl}), and Theorem~\ref{rot} (in 
\S\ref{sd}), we will need the obvious equality
\begin{equation}\label{ieh}
\ie(\psi)\,=\,\ve\hs\|\psi\|_{2,1}^2+I_0^0(\psi)\,
+\,\langle(\chi-\ve)\psi,\psi\rangle_2^{\phantom i}\hskip10pt\text{\rm for}
\hskip7pt\psi\in L^2_1M\hh,
\end{equation}
where $\,\langle\,,\rangle_2^{\phantom i}$ is the inner product of 
$\,L^2\hskip-1.3ptM$, along with the (well known)
\begin{lemma}\label{pdczb}If a sequence\/ $\,\varphi_j^{\phantom i}$, 
$\,j=1,2,3,\dots\hs$, in the So\-bo\-lev space\/ $\,L^2_1M\,$ of a compact 
Riemannian manifold\/ $\,(M,g)\,$ of dimension\/ $\,n\ge3$ \hbox{is bounded in 
the So\-bo\-lev} norm\/ $\,\|\hskip2.3pt\|_{2,1}^{\phantom i}$, then, after\/ 
$\,\varphi_j^{\phantom i}$ has been replaced by a suitable sub\-se\-quence, 
there will exist a function\/ $\,\psi\in L^2_1M\,$ such that, as\/ 
$\,j\to\infty$, one has simultaneously the convergences\/ 
$\,\varphi_j^{\phantom i}\to\psi\,$ in the\/ $\,L^2$ and\/ $\,L^r$ norms, with 
any fixed\/ $\,r\in(2,2n/(n-2))$, the weak convergence\/ 
$\,\varphi_j^{\phantom i}\to\psi\,$ in the norm\/ 
$\,\|\hskip2.3pt\|_{2,1}^{\phantom i}$, and convergence of the norms\/ 
$\,\|\varphi_j^{\phantom i}\|_{2,1}^{\phantom i}$ to some real number\/ 
$\,\gamma\ge\|\psi\|_{2,1}^{\phantom i}$.
\end{lemma}
\begin{proof}[Proof]Except for the last inequality 
$\,\gamma\ge\|\psi\|_{2,1}^{\phantom i}$, the existence of a sub\-se\-quence 
with the required properties is an obvious consequence of the 
Rel\-li\-ch-Kon\-dra\-shov and Ba\-na\-ch-A\-la\-o\-glu theorems; note that 
$\,L^2_1M\,$ is a Hil\-ber\-t space. On the other hand, the Schwarz inequality 
$\,|\langle\psi,\varphi_j^{\phantom i}\rangle_{2,1}^{\phantom i}|\le
\|\psi\|_{2,1}^{\phantom i}\|\varphi_j^{\phantom i}\|_{2,1}^{\phantom i}$ in 
$\,L^2_1M\,$ and the weak convergence $\,\varphi_j^{\phantom i}\to\psi$ imply, 
in the limit, that 
$\,\|\psi\|_{2,1}^2\le\gamma\hs\|\psi\|_{2,1}^{\phantom i}$ and either 
$\,\|\psi\|_{2,1}^{\phantom i}=0\le\gamma$, or 
$\,\|\psi\|_{2,1}^{\phantom i}\ne0\,$ and both sides may be divided by 
$\,\|\psi\|_{2,1}^{\phantom i}$.
\end{proof}

\section{Proof of Lemma~\ref{logns}}\label{dl}
Setting $\,H(\varphi)=\varphi\hh^2\log\hh|\varphi|\,$ if 
$\,\varphi\in\bbR\smallsetminus\{0\}\,$ and $\,H(0)=0$, we obtain a function 
$\,H\,$ of the variable $\,\varphi\in\bbR\,$ having a continuous derivative 
$\,H'\nnh$. With any fixed $\,r\in(2,\infty)$, both $\,H(\varphi)/|\varphi|^r$ 
and $\,H'(\varphi)/|\varphi|^{r-1}$ tend to zero as $\,|\varphi|\to\infty$, 
leading to the estimates $\,|H(\varphi)|\le c\hs(1+|\varphi|^r)]\,$ and 
$\,|H'(\varphi)|\le c\,\mathrm{max}\,(1,|\varphi|^{r-1})\,$ for 
$\,\varphi\in\bbR$, where $\,c>0\,$ depends only on $\,r$. 

Let $\,\varphi,\psi\in\bbR\,$ and $\,r\in(2,\infty)$. Since 
$\,|H'(\varphi)|\le c\,\mathrm{max}\,(1,|\varphi|^{r-1})$, La\-gran\-ge's 
classical mean value theorem implies that 
$\,|H(\varphi)-H(\psi)|
\le c\hs|\varphi-\psi|\,\mathrm{max}\,(1,|\zeta|^{r-1})$, for some 
$\,\zeta\,$ lying between $\,\varphi\,$ and $\,\psi$. Thus, 
$\,|H(\varphi)-H(\psi)|
\le c\hs|\varphi-\psi|\,\mathrm{max}\,(1,|\varphi|^{r-1}\nnh,|\psi|^{r-1})$. 
The H\"ol\-der inequality, with $\,q\,$ given by 
$\,r^{-1}\nh+q^{-1}\nh=1$, that is, $\,q=r/(r-1)$, now yields, for {\it 
functions\/} $\,\varphi,\psi\in L^2_1M\,$ and $\,r\in(2,\infty)$, the integral 
estimate
\begin{equation}\label{osz}
\|H(\varphi)-H(\psi)\|_1^{\phantom i}
\le c\hs\|\varphi-\psi\|_r^{\phantom i}\hs
(1+\|\varphi\|_r^r+\|\psi\|_r^r)^{1/q},
\end{equation}
in which $\,c\,$ depends on $\,r\,$ (and the right-hand side may be infinite).

Let $\,r=2n/(n-2)$. The relation $\,|H(\varphi)|\le c\hs(1+|\varphi|^r)$, and 
the inclusion $\,L^2_1M\subset L^r\hskip-2ptM\,$ in (\ref{nso}) imply 
integrability of $\,|\varphi|^r$ for $\,\varphi\in L^2_1M$, proving (a). Since 
$\,|I_0^0(\varphi)-I_0^0(\psi)|=\|H(\varphi)-H(\psi)\|_1^{\phantom i}$, while 
$\,\|\psi\|_r^{\phantom i}\le\|\varphi\|_r^{\phantom i}
+\|\varphi-\psi\|_r^{\phantom i}$, convergence of $\,\psi\,$ to $\,\varphi\,$ 
in the So\-bo\-lev norm $\,\|\hskip2.3pt\|_{2,1}^{\phantom i}$ implies, via 
(\ref{nso}), convergence in the norm $\,\|\hskip2.3pt\|_r^{\phantom i}$ and, 
consequently, also the relation $\,I_0^0(\varphi)\to I_0^0(\psi)$. The 
functional $\,I_0^0:L^2_1M\to\bbR\,$ is thus continuous, and (b) easily 
follows in view of (\ref{ieh}).

To obtain (c), we begin with convexity of the exponential function, that is, 
the {\it  Jen\-se\-n inequality}, which states that 
$\,\textstyle{\int_M}F\,d\mu\,\le\,\log\textstyle{\int_M}e^Fd\mu$ for 
any integrable function $\,F:M\to\bbR\,$ on a space $\,M\,$ carrying a 
probability measure $\,\mu$. Proof: it suffices to show this in the case of 
simple functions or, equivalently, verify that 
$\,q_1^{c_1}\nnh\ldots\,\hs q_k^{c_k}
\le c_1^{\phantom i}q_1^{\phantom i}+\ldots+c_k^{\phantom i}q_k^{\phantom i}$ 
whenever $\,q_j^{\phantom i}\in(0,\infty)\,$ and 
$\,c_j^{\phantom i}\in[\hs0,\infty)$, $\,j=1,\dots,k$, with 
$\,\sum_{j=1}^kc_j^{\phantom i}=1$, which is easily achieved by applying 
$\,\partial/\partial q_1^{\phantom i}$ to maximize the difference 
$\,q_1^{c_1}\nnh\ldots\hs q_k^{c_k}\nnh-c_1^{\phantom i}q_1^{\phantom i}-\ldots
-c_k^{\phantom i}q_k^{\phantom i}$ for fixed 
$\,q_2^{\phantom i},\dots,q_k^{\phantom i}$ and 
$\,c_1^{\phantom i},\dots,c_k^{\phantom i}$.

Obviously, $\,a\nh\int_M\varphi\hh^2\log\varphi\,\mathrm{d}g
=\int_M\varphi\hh^2\log\varphi^a\hs\mathrm{d}g\,$ if 
$\,a,\ve\in(0,\infty)\,$ and $\,\varphi\in\varSigma\,$ (notation as in 
(c)). Thus, $\,\int_M\varphi\hh^2\log\varphi^a\hs\mathrm{d}g
\le(2+a)\,\log\|\varphi\|_{2+a}^{\phantom i}$ from Jen\-sen's inequality  
$\,\textstyle{\int_M}F\,d\mu\,\le\,\log\textstyle{\int_M}e^Fd\mu\,$ for 
$\,d\mu=\varphi\hh^2\mathrm{d}g\,$ and $\,F=\log\varphi^a$. Setting 
$\,a=4/(n-2)\,$ and choosing $\,p=2\,$ in the So\-bo\-lev inequality 
(\ref{nso}), we get $\,2\int_M\varphi\hh^2\log\varphi\,\mathrm{d}g
\le n\log\hs C\|\varphi\|_{2,1}^{\phantom i}$. Since 
$\,\|\varphi\|_{2,1}^2=\int_M(|\nabla\varphi|^2+\varphi\hh^2)\,\mathrm{d}g$, 
this last inequality yields 
$\,\ie(\varphi)\ge\varPhi(\xi)+\mathrm{min}\hskip2.3pt\chi$, 
where $\,\varPhi(\xi)=\ve\hs(\xi^2-1)-(1+2/a)\log\hs C\xi\,$ for 
$\,\xi=\|\varphi\|_{2,1}^{\phantom i}\ge\|\varphi\|_2^{\phantom i}=1$. 
In addition, $\,\inf\hs\{\varPhi(\xi):\xi\in[\hs1,\infty)\}\hs>\hs-\infty$, so 
that, {\it if\/ $\,\ve>0$, the functional\/ $\,\ie$ is bounded from below on 
the set\/} $\,\varSigma$.

Let $\,\ve>0$. Using the above italicized conclusion, we denote by 
$\,\kappa$ the infimum of the functional $\,\ie$ on $\,\varSigma\,$ and fix 
a sequence $\,\varphi_j^{\phantom i}\in\varSigma$, $\,j=1,2,3,\dots\hs$, for 
which $\,\ie(\varphi_j^{\phantom i})\to\kappa\,$ as $\,j\to\infty$. The 
sequence $\,\varphi_j^{\phantom i}$ is thus bounded in the So\-bo\-lev norm  
$\,\|\hskip2.3pt\|_{2,1}^{\phantom i}$, since in the obvious equality 
$\,\ve\hs\|\nabla\varphi\|_2^{\phantom i}
=2\ie(\varphi)-2I_{\ve/2}^\chi(\varphi)$, for 
$\,\varphi=\varphi_j^{\phantom i}$, the terms $\,2\ie(\varphi)\,$ converge, 
and $\,2I_{\ve/2}^\chi(\varphi)\,$ are bounded from below (as we just saw). 
Let us now replace the sequence $\,\varphi_j^{\phantom i}$ with a 
sub\-se\-quence chosen as in Lemma~\ref{pdczb} for some real $\,\gamma\,$ and 
a limit function in $\,L^2_1M$, denoted by $\,\varphi\,$ (rather than 
$\,\psi$). The estimate (\ref{osz}) and the inclusion 
$\,L^2_1M\subset L^r\hskip-2ptM$ in (\ref{nso}) give 
$\,I_0^0(\varphi_j^{\phantom i})\to I_0^0(\varphi)$. As $\,j\to\infty$, 
(\ref{ieh}) for $\,\psi=\varphi_j^{\phantom i}$ therefore yields 
$\,\kappa=\ve\hs\gamma+I_0^0(\varphi)
+\langle(\chi-\ve)\varphi,\varphi\rangle_2^{\phantom i}
=\ie(\varphi)+\ve\hh(\gamma^2\nh-\|\varphi\|_{2,1}^2)$. Consequently,
\begin{equation}\label{kmi}
\kappa\,-\,\ie(\varphi)\,=\,\ve\hh(\gamma^2\nh-\|\varphi\|_{2,1}^2)\hh.
\end{equation}
Since $\,\kappa\le\ie(\varphi)$, we have 
$\,\gamma\le\|\varphi\|_{2,1}^{\phantom i}$ while, by Lemma~\ref{pdczb}, 
$\,\gamma\ge\|\varphi\|_{2,1}^{\phantom i}$, and so 
$\,\gamma=\|\varphi\|_{2,1}^{\phantom i}$. Thus, (\ref{kmi}) gives 
$\,\kappa=\ie(\varphi)$, proving (c).

Let us now fix $\,\varphi\in\varSigma\,$ such that $\,\ie(\varphi)\,$ has the 
minimum value $\,\kappa$. If $\,\psi\in\varSigma\,$ and 
$\,\langle\varphi,\psi\rangle_2^{\phantom i}=0$, putting 
$\,\psi_\theta^{\phantom i}=(\cos\theta)\varphi+(\sin\theta)\psi$, for 
$\,\theta\in\bbR$, we obtain a curve 
$\,\theta\mapsto\psi_\theta^{\phantom i}\in\varSigma$. The function  
$\,\theta\mapsto\ie(\psi_\theta^{\phantom i})\,$ is then differentiable and 
its derivative can be evaluated by differentiation under the integral symbol, 
yielding $\,d\hs[\ie(\psi_\theta^{\phantom i})]/d\theta
=\int_M(\partial\varPi_\theta^{\phantom i}/\partial\theta)
\,\mathrm{d}g$, where $\,\varPi_\theta^{\phantom i}$ denotes the integrand in 
(\ref{fnc}) with $\,\varphi\,$ replaced by $\,\psi_\theta^{\phantom i}$. 

This is immediate from Le\-besgue's dominated convergence theorem, since the 
absolute value $\,|\partial\varPi_\theta^{\phantom i}/\partial\theta|\,$ is 
bounded from above, uniformly in $\,\theta$, by an integrable function. In 
fact, the first and third terms in $\,\varPi_\theta^{\phantom i}$, 
differentiated with respect to $\,\theta$, yield a linear combination of 
$\,\cos2\theta$ and $\,\sin2\theta\,$ with coefficients that are integrable 
functions; note that due to the definition of $\,L^p_1M$, preceding formula 
(\ref{nso}), the distributional gradient $\,\nabla\psi\,$ of any function 
$\,\psi\in L^p_1M\,$ is a measurable vector field with a square-integrable 
$\,g$-norm $\,|\nabla\psi|:M\to\bbR$. The derivative with respect to 
$\,\theta\,$ of the second term in $\,\varPi_\theta^{\phantom i}$ has in turn 
the absolute value $\,|\partial\psi_\theta^{\phantom i}/\partial\theta|\,
|H'(\psi_\theta^{\phantom i})|$, for $\,H\,$ as at the beginning of the proof; 
the inequality 
$\,|H'(\varphi)|\le c\,\mathrm{max}\,(1,|\varphi|^{r-1})$ established there, 
along with the obvious fact that $\,|\psi_\theta^{\phantom i}|\,$ and 
$\,|\partial\psi_\theta^{\phantom i}/\partial\theta|\,$ do not exceed 
$\,2\hskip1.4pt\mathrm{max}\,(|\varphi|,|\psi|)$, gives
\[
|\partial\psi_\theta^{\phantom i}/\partial\theta|\,
|H'(\psi_\theta^{\phantom i})|\,
\le\,c\,\mathrm{max}\,(1,|\varphi|^r\nnh,|\psi|^r)\,
\le\,c\hs(1+|\varphi|^r\nh+|\psi|^r)
\]
with a new constant $\,c>0$. For $\,r=2n/(n-2)\,$ we now note that the 
inclusion $\,L^2_1M\subset L^r\hskip-2ptM\,$ in (\ref{nso}) yields 
integrability of $\,|\varphi|^r$ and $\,|\psi|^r\nnh$.

As $\,\varphi\,$ minimizes $\,\ie\,$ on $\,\varSigma$, the derivative 
$\,d\hs[\ie(\psi_\theta^{\phantom i})]/d\theta\,$ at $\,\theta=0$ equals 
$\,0$. Differentiating the integrand, one consequently gets 
\begin{equation}\label{nfn}
-\hh\ve\hh\langle\nabla\varphi,\nabla\psi\rangle_2^{\phantom i}\,
+\,\,\langle\varphi\hs\log\hh|\varphi|\hs
+\hs(\kappa-\chi)\hs\varphi,\hs\psi\rangle_2^{\phantom i}\,=\,0
\end{equation}
for $\,\psi\in L^2_1M\,$ such that 
$\,\langle\varphi,\psi\rangle_2^{\phantom i}=0$. On the other hand, 
(\ref{nfn}) also holds for $\,\psi=\varphi$, since $\,\ie(\varphi)=\kappa$. We 
thus have (\ref{nfn}) whenever $\,\psi\in L^2_1M\,$ and -- in particular -- 
for test functions $\,\psi\in\mathcal{F}M$, which proves (\ref{dys}).

\section{Outline of the proof of Rothaus's theorem (Thm.~\ref{rot})}\label{sd}
We need another, well-known
\begin{lemma}\label{wbzwz}Let\/ $\,\varphi\,$ be any function in the 
So\-bo\-lev space\/ $\,L^2_1M$ of a compact Riemannian manifold\/ $\,(M,g)\,$ 
of dimension\/ $\,n\ge3$. The function\/ $\,\psi=|\varphi|\,$ then satisfies 
the conditions\/ $\,\psi\in L^2_1M\,$ and\/ 
$\,\|\psi\|_{2,1}^{\phantom i}\le\|\varphi\|_{2,1}^{\phantom i}$.
\end{lemma}
\begin{proof}[Proof]Suppose first that $\,\varphi\,$ is smooth, and denote by 
$\,\theta\,$ any term of a fixed sequence of  positive numbers tending to 
$\,0$. For  the sequence 
$\,\varphi_\theta^{\phantom i}=\sqrt{\varphi^2\nh+\theta^2\,}\,$ of positive 
smooth functions, the obvious equality 
$\,\varphi_\theta^{\phantom i}-\psi
=\theta^2/(\varphi_\theta^{\phantom i}+\psi)\,$ gives 
$\,|\varphi_\theta^{\phantom i}-\psi|\le\theta$. One thus has uniform 
convergence $\,\varphi_\theta^{\phantom i}\to\psi$, which also yields 
$\,\|\varphi_\theta^{\phantom i}-\psi\|_2^{\phantom i}\to0\,$ and 
$\,\|\varphi_\theta^{\phantom i}\|_2^{\phantom i}\to\|\psi\|_2^{\phantom i}$. 
On the other hand, $\,\nabla\varphi_\theta^{\phantom i}
=(\varphi/\varphi_\theta^{\phantom i})\nabla\varphi$, while 
$\,|\varphi/\varphi_\theta^{\phantom i}|\le1$. so that 
$\,|\nabla\varphi_\theta^{\phantom i}|\le|\nabla\varphi|$ and 
$\,\|\nabla\varphi_\theta^{\phantom i}\|_2^{\phantom i}
\le\|\nabla\varphi\|_2^{\phantom i}$. Replacing our sequence with a suitable 
sub\-se\-quence, we obtain 
$\,\|\nabla\varphi_\theta^{\phantom i}\|_2^{\phantom i}\to c\,$ for some 
$\,c\le\|\nabla\varphi\|_2^{\phantom i}$. The sequence 
$\,\|\varphi_\theta^{\phantom i}\|_{2,1}^2
=\|\nabla\varphi_\theta^{\phantom i}\|_2^2
+\|\varphi_\theta^{\phantom i}\|_2^2$ thus converges to $\,\gamma^2\nnh$, for 
$\,\gamma=(\|\psi\|_2^2+c^2)^{1/2}\nnh$. In view of Lemma~\ref{pdczb}, 
replacing our sub\-se\-quence by a further sub\-se\-quence allows us to 
assume convergence of $\,\varphi_\theta^{\phantom i}$ in the norm $\,L^2$ to 
some limit function lying in $\,L^2_1M$. Since we already know that 
$\,\|\varphi_\theta^{\phantom i}-\psi\|_2^{\phantom i}\to0$, this limit 
function must be $\,\psi=|\varphi|$, which implies that 
$\,\psi\in L^2_1M$. The inequalities $\,\gamma\ge\|\psi\|_{2,1}^{\phantom i}$ 
(in Lemma~\ref{pdczb}) and $\,c\le\|\nabla\varphi\|_2^{\phantom i}$ in turn 
show that $\,\|\psi\|_{2,1}^2\le\gamma^2\nh=\|\psi\|_2^2+c^2\nh
\le\|\psi\|_2^2+\|\nabla\varphi\|_2^{\phantom i}=\|\varphi\|_{2,1}^2$, proving 
our claim in the case where $\,\varphi\,$ is smooth.

For any function $\,\varphi\in L^2_1M\,$ we have 
$\,\varphi_j^{\phantom i}\to\varphi\,$ as $\,j\to\infty$, in the norm 
$\,\|\hskip2.3pt\|_{2,1}^{\phantom i}$, with some sequence 
$\,\varphi_j^{\phantom i}$, $\,j=1,2,3,\dots\hs$, of smooth functions. Let 
$\,\psi_j^{\phantom i}=|\varphi_j^{\phantom i}|\,$ and $\,\psi=|\varphi|$. 
Thus, $\,\psi_j^{\phantom i}\in L^2_1M$, and the sequence 
$\,\|\psi_j^{\phantom i}\|_{2,1}^{\phantom i}
\le\|\varphi_j^{\phantom i}\|_{2,1}^{\phantom i}$ of norms is bounded; 
replacing it by a convergent sub\-se\-quence, we obtain 
$\,\|\psi_j^{\phantom i}\|_{2,1}^{\phantom i}\to\gamma
\le\|\varphi\|_{2,1}^{\phantom i}$ (since 
$\,\|\varphi_j^{\phantom i}\|_{2,1}^{\phantom i}
\to\|\varphi\|_{2,1}^{\phantom i}$). Lemma~\ref{pdczb} for the sequence 
$\,\psi_j^{\phantom i}$ allows us to choose a limit function in $\,L^2_1M$, 
which must coincide with $\,\psi\,$ due to the convergence 
$\,\psi_j^{\phantom i}\to\psi\,$ in the $\,L^2$ norm (obvious in view of the 
convergence $\,\varphi_j^{\phantom i}\to\varphi\,$ in $\,L^2\hskip-1.3ptM$, 
as $\,|\psi_j^{\phantom i}-\psi|\le|\varphi_j^{\phantom i}-\varphi|$), while 
the inequality $\,\|\psi\|_{2,1}^{\phantom i}\le\gamma\,$ in Lemma~\ref{pdczb} 
gives $\,\|\psi\|_{2,1}^{\phantom i}\le\|\varphi\|_{2,1}^{\phantom i}$.
\end{proof}
With the same assumptions and notations as in Lemma~\ref{logns}, given 
$\,\ve>0$, let $\,\varphi$ minimize $\,\ie\,$ on the set $\,\varSigma$. Then 
$\,\psi=|\varphi|\,$ also minimizes $\,\ie\,$ on $\,\varSigma$. Namely, 
Lemma~\ref{wbzwz} shows that $\,\psi\in\varSigma\,$ and $\,\psi\,$ satisfies 
the inequality 
$\,\|\psi\|_{2,1}^{\phantom i}\le\|\varphi\|_{2,1}^{\phantom i}$, while 
$\,\ie(\psi)\le\ie(\varphi)\,$ as a consequence of this inequality and 
(\ref{ieh}), since the two final terms in (\ref{ieh}) remain unchanged if 
$\,\psi\,$ is replaced by $\,|\psi|$.

In other words, we may also assume that the function $\,\varphi\,$ 
minimizing $\,\ie\,$ on $\,\varSigma\,$ is nonnegative. It then follows that 
$\,\varphi$ is positive everywhere and smooth. The reasons for the last 
conclusion can only be outlined here. First, using the De Giorgi and Nash 
method, one shows that $\,\varphi\,$ is of class $\,C^2$ with 
lo\-cal\-ly-H\"ol\-der second partial derivatives. Next, a suitable version 
of a local maximum principle, using geodesic coordinates, proves that the 
zero set of $\,\varphi\,$ is open. Its openness -- along with the condition 
$\,\|\varphi\|_2^{\phantom i}=1\,$ (which is a part of the definition of 
$\,\varSigma$, and prevents $\,\varphi\,$ from vanishing identically) -- 
implies positivity of $\,\varphi$. A boot\-strap\-ping-type argument in spaces 
of multiply differentiable functions with H\"ol\-der derivatives yields in 
turn smoothness of $\,\varphi$.
\begin{proof}[Proof of Theorem~\ref{rot}]Let us fix $\,(M,g)\,$ and 
$\,\lambda\,$ satisfying the hypotheses of the theorem, and a smooth function 
$\,\psi:M\to\mathrm{I\!R}$. For $\,\ve=1/\lambda\,$ and 
$\,\chi=-\psi/(2\lambda)\,$ there exists -- as we saw earlier -- a positive 
smooth solution $\,\varphi\,$ of (\ref{dys}) with some constant $\,\kappa$. 
Setting $\,f=-2(\kappa+\log\hh\varphi)$, one easily verifies that (\ref{dys}) 
takes the form $\,\mathcal{R}f+\hh\lambda f=\psi$.
\end{proof}

\section{Proof of Perelman's gradient-type theorem (Thm.~\ref{kzs})}\label{dt}

\vskip4pt
\subsection*{Assumption.} $(M,g)\,$ is a compact Riemannian manifold and 
(\ref{sol}) holds for a fixed real number $\,\lambda\,$ and a fixed vector 
field $\,w$.

\vskip4pt
\subsection*{Objective.} To find a function $\,f\,$ satisfying equation 
(\ref{ndf}):  
$\,\nabla\df+\hh\mathrm{Ric}\hh=\hskip.7pt\lambda\hskip.7pt g$. We may assume 
here that $\,n=\dim M\ge3\,$ and the sol\-i\-ton constant $\,\lambda\,$ is 
positive, since -- as shown by Hamilton \cite{hamilton-tr} and, back in 1974, 
Bourguignon \cite{bourguignon} -- in the compact case condition (\ref{sol}) 
with $\,n=2$, or with $\,\lambda\le0$, always gives (\ref{ein}), that is, 
(\ref{ndf}) for $\,f=0$.

\vskip4pt
\subsection*{How to proceed.} The function $\,f\,$ is supposed to satisfy 
(\ref{ndf}); where to get it from? Lemma~\ref{stale} suggests an answer -- 
three functions, naturally associated with $\,f\,$ should, {\it ex post 
facto}, turn out to be constant, which (with the right choice of their 
constant values) leads to three sec\-ond-or\-der elliptic equations, imposed 
on the required function $\,f$. One can thus try to show first that one of the 
three equations has a solution $\,f$, and then -- that such $\,f\,$ also 
satisfies (\ref{ndf}). In the first equation, the constant in question must 
obviously be the mean value $\,\mathrm{s}_{\hs\mathrm{avg}}$ of the scalar 
curvature $\,\mathrm{s}\,$ (since $\,\int_M\Delta f\,\mathrm{d}g=0$), and 
$\,f\,$ such that 
$\,\Delta\ef+\hs\mathrm{s}\hs=\hs\mathrm{s}_{\hs\mathrm{avg}}$ exists due to 
the general solvability criterion for linear elliptic equations. However, 
nobody knows how to get from here to (\ref{ndf}). Things look better for the 
third equation, in which, since $\,\lc>0$, by adding to $\,f\,$ a suitable 
constant we may require that the constant on the right-hand side be zero. As 
we will see, solvability follows here from Rothaus's theorem.

\vskip4pt
\subsection*{Argument.} For any compact Riemannian manifold $\,(M,g)$, any 
function $\,f$, constant $\,\lambda\,$ and vector field $\,w$, let us set
\[
h=\nabla\df+\hh\mathrm{Ric}\hh-\hskip.7pt\lambda\hskip.7pt g\hh,\hskip9pt
\bz=\pounds\nnh_w\phantom{^i}g
+\hh\mathrm{Ric}\hh-\hskip.7pt\lambda\hskip.7pt g\hh,
\hskip9pt
\psi=\Delta e^{-f}+2\hh\dv\hh[e^{-f}\hskip-.7pt w]\hskip.7pt,
\]
where $\,\dv\,$ is the divergence operator acting on vector fields. Then (with 
no further assumptions!), for the operator $\,\mathcal{R}\,$ given by 
(\ref{rfe}),
\[
\displaystyle{\int_M|h|^2\hskip.7pt e^{-f}\mathrm{d}g}\hskip.7pt\,\,
+\hskip2pt\displaystyle{\int_M(\mathcal{R}f+\hh \lambda f
+\hskip.7pt\mathrm{s}/2)\hskip.7pt\psi\,\mathrm{d}g}\,\,
=\hskip2pt\displaystyle{\int_M\langle h,\bz\rangle\hskip.7pt 
e^{-f}\mathrm{d}g}\hskip.7pt.
\]
(Proof -- a trivial, though tedious, integration by parts.) In our situation, 
with $\,\lambda\,$ and $\,w\,$ satisfying (\ref{sol}), the integral on the 
right-hand side vanishes, since $\,\bz=0$, and Rothaus's theorem 
(Thm.~\ref{rot}) allows us to choose $\,f$ such that 
$\,\mathcal{R}f+\hh\lambda f+\hs\mathrm{s}/2\hs=\hs0$. Thus, $\,h=0$, which 
completes the proof.

\section{Canonical K\"ah\-ler metrics}\label{km}
Let $\,g\,$ be a K\"ah\-ler metric on a fixed compact complex manifold $\,M$. 
The formulae $\,\omega=g(J\,\cdot\,,\,\cdot\,)\,$ and 
$\,\rho=\mathrm{Ric}\hh(J\,\cdot\,,\,\cdot\,)\,$ then define the {\it 
K\"ah\-ler form\/} and {\it Ric\-ci form\/} of the metric $\,g$. Both of them 
are closed $\,2$-forms, that is, skew-sym\-met\-ric $\,2$-ten\-sor fields (due 
to (\ref{grh})) with $\,d\hs\omega=d\hh\rho=0$. If a function $\,f\,$ on 
$\,M\,$ satisfies the grad\-i\-ent-sol\-i\-ton equation (\ref{ndf}), then 
$\,i\hskip1pt\partial\overline{\partial}\hskip-.9ptf+\hs\rho\hs\,
=\,\lc\hs\omega$, since the form  
$\,i\hskip1pt\partial\overline{\partial}\hskip-.9ptf\,$ is exact (being the 
exterior derivative of $\,i\hskip1pt\overline{\partial}\hskip-.9ptf\hh$). This 
implies equality of de Rham co\-ho\-mo\-lo\-gy classes: 
$\,[\hs\rho\hh]=\lc\hs[\omega\hh]\in H^2(M,\bbR)$. On the other hand, 
$\,[\hs\rho\hh]\,$ is always equal to the first Chern class 
$\,c_1^{\phantom j}$ of $\,M\,$ multiplied by $\,2\pi$, so that it only 
depends on the complex structure -- and not, for instance, on $\,g$. 

For a compact complex manifold to admit a K\"ah\-ler-Ric\-ci sol\-i\-ton with 
a positive (or negative) sol\-i\-ton constant $\,\lambda\,$ it is thus 
necessary that its first Chern class $\,c_1^{\phantom j}$ be positive (or 
negative), in the sense of realizability of $\,c_1^{\phantom j}$ (or 
$\,-\hh c_1^{\phantom j}$) as the co\-ho\-mo\-lo\-gy class of the K\"ah\-ler 
form of some K\"ah\-ler metric.

In complex dimension $\,2\,$ this necessary condition is also sufficient.  
In addition, on compact complex surfaces with positive or negative first Chern 
class $\,c_1^{\phantom j}$, K\"ah\-ler-Ric\-ci sol\-i\-tons constitute a 
natural choice of a {\it distinguished}, or {\it canonical}, K\"ah\-ler 
metric. 

More precisely, {\it such a surface always admits a K\"ah\-ler-Ric\-ci 
sol\-i\-ton, which is in addition unique up to the action of the identity 
component of its complex automorphism group, and rescalings.}

The above statement summarizes a series of results, both classical and more 
recent. Here belong the theorems establishing:
  \begin{equivalence}
\item uniqueness of K\"ah\-ler-Ein\-stein metrics (1957: Eugenio Calabi 
\cite{calabi}, for $\,c_1^{\phantom j}<0$, 1987: Shigetoshi Bando, Toshiki 
Mabuchi \cite{bando-mabuchi}, for $\,c_1^{\phantom j}>0$),
\item truth of Calabi's conjecture about the existence of a 
K\"ah\-ler-Ein\-stein metric when $\,c_1^{\phantom j}<0\,$ (1978: Thierry 
Aubin \cite{aubin}, Shing-Tung Yau \cite{yau}),
\item existence K\"ah\-ler-Ric\-ci sol\-i\-tons on compact toric complex 
manifolds with $\,c_1^{\phantom j}>0\,$ (2004: Xu-Jia Wang, Xiaohua Zhu 
\cite{wang-zhu}),
\item uniqueness of a K\"ah\-ler-Ric\-ci sol\-i\-ton when 
$\,c_1^{\phantom j}>0$, modulo automorphisms from the identity component 
(2002: Tian and Zhu \cite{tian-zhu}).
  \end{equivalence}

\section{Ric\-ci-Hes\-i\-an equations}\label{rh}
From now on all functions are -- by definition -- {\it smooth}.

Using Maschler's terminology \cite{maschler}, we say that a function $\,\kp\,$ 
on a Riemannian manifold $\,(M,g)\,$ {\it satisfies a Ric\-ci-Hess\-i\-an 
equation\/} if, for some function $\,\alpha\,$ on $\,M\nh$, nonzero at all 
points of a dense subset, one has, with the notation introduced in (\ref{trz}),
\begin{equation}\label{rhe}
\{\alpha\nabla d\kp+\hs\mathrm{Ric}\}_0^{\phantom j}\hh=\,0
\end{equation}
or, equivalently: 
$\,\alpha\nabla d\kp+\hh\mathrm{Ric}\hh=\hh\eta\hs g\,$ for some function 
$\,\eta$.

Solutions of equations of type (\ref{rhe}) exist, for instance, on certain 
K\"ah\-ler manifolds, forming the family of {\it special K\"ah\-ler-Ric\-ci 
potentials}, which is completely classified, both locally 
\cite[\S18]{derdzinski-maschler-lc} and in the compact case 
\cite[\S16]{derdzinski-maschler-sk}. Their definition and a discussion of how 
they are related to condition (\ref{rhe}) can be found in 
\cite[\S7]{derdzinski-maschler-lc} and \cite{maschler}.

Another special case of a Ric\-ci-Hess\-i\-an equation (\ref{rhe}) is the 
grad\-i\-ent-sol\-i\-ton equation (\ref{ndf}), in which $\,\kp=f\,$ and 
$\,\alpha=1$. A further connection between equations (\ref{rhe}) and 
(\ref{ndf}) is due to the fact that many solutions of Ric\-ci-Hess\-i\-an 
equations (\ref{rhe}), which themselves do not satisfy (\ref{ndf}), may be 
used to construct solutions of (\ref{ndf}) by suitably modifying the metric 
$\,g\,$ and the functions appearing in (\ref{rhe}). The modifications of 
metrics consist here in their {\it con\-for\-mal changes}, a detailed 
discussion of which will be given later (\S\ref{zk}).

The approach just described was developed by Maschler in \cite{maschler}. 
His initial solutions $\,\kp\,$ of equations of type (\ref{rhe}) belonged to 
the class -- mentioned earlier in this section -- of special 
K\"ah\-ler-Ric\-ci potentials.

The following lemma, using the convention borrowed from the end of \S\ref{pr}, 
is a slight variation on Maschler's argument \cite[Remark 4.2]{maschler}. 
The solution $\,\kp\,$ of a Ric\-ci-Hess\-i\-an equation (\ref{rhe}), arising 
here, is a special K\"ah\-ler-Ric\-ci potential on the compact K\"ah\-ler 
manifold $\,(M^m_k\nnh,g)$, that is, belongs to a family of examples 
constructed in \cite[\S5]{derdzinski-maschler-sk}.
\begin{lemma}\label{rhxpl}Let\/ $\hs m,k\in\bbZ\hh$ and smooth functions\/ 
$\hs\jx,\fy\nh:\nh[\hs0,\tc\hh]\nh\to\nh\bbR\,$ of the variable\/ $\,\jt$, 
with\/ $\,\tc\in(0,\infty)$, satisfy the conditions\/ $\,m\hn>k\hn>0\,$ and
  \begin{equivalence}
\item $\dd=(m-1)\hs\dot\jx\dt+m\hh\fy-m$,
\item $\fy(0)=\fy(\tc)=0$, while\/ $\,\fy>0\,$ on the interval\/ $\,(0,\tc)$,
\item $\dt(0)=k\,$ and\/ $\,\dt(\tc)=-\hh k$,
  \end{equivalence}
where\/ $\,(\,\,)\dot{\,}=\,d/d\jt$. Then the complex manifold\/ 
$\,M=M^m_k$ described by\/ $(*)$ in\/ \S\ref{sk} admits a K\"ah\-ler metric\/ 
$\,g\,$ with scalar curvature\/ $\,\mathrm{s}\,$ and a smooth surjective 
function\/ $\,\jt:M\to[\hs0,\tc\hh]\,$ with the following properties\/{\rm:}
  \begin{alternative}
\item the function\/ $\,\kp=e^\jt$ is a solution of the Ric\-ci-Hess\-i\-an 
equation\/ {\rm(\ref{rhe})}, that is, 
$\,\{\alpha\nabla d\kp+\hs\mathrm{Ric}\}_0^{\phantom j}\nh=0$, where\/ 
$\,\alpha=(m-1)(\dot\jx+1)\hh e^{-\jt}\nnh$,
\item $\Delta\kp=2(\dt+m\hh\fy)\,$ and\/ 
$\,g(\nabla\kp,\nh\nabla\kp)=2e^\jt\fy$, for the function\/ $\,\kp=e^\jt\nnh$,
\item $\,e^\jt\mathrm{s}/2\,\,=\,\,m(m-1)\,-\,m(m-1)\fy\,-\,(2m-1)\dt\,-\,\dd$.
   \end{alternative}
\end{lemma}
\begin{proof}[Outline of proof]The metric $\,g\,$ and function $\,\jt\,$ will 
be shown to exist via an explicit construction, carried out in a slightly 
larger class of complex manifolds than just the family $\,M^m_k$ (see also the 
comment at the end of \S\ref{sk}). To be specific, suppose that $\,N\hs$ is a 
compact complex manifold of dimension $\,m-1\,$ and its canonical bundle (the 
top exterior power of cotangent bundle) can be raised to the fractional 
tensor power with the exponent $\,k/m$. When $\,m\,$ and $\,k\,$ are 
relatively prime, this amounts to realizability of the canonical bundle as 
the $\,m\hh$th tensor power of some line bundle; it is known 
\cite{kobayashi-ochiai} that $\,N\hs$ must then be bi\-hol\-o\-mor\-phic to 
the projective space $\,\bbCP^{m-1}\nnh$, for which the tautological bundle 
is an $\,m\hh$th tensor root of the canonical bundle. If, however, the 
fraction $\,k/m\,$ can be simplified, there are more such examples, as 
illustrated by the case of odd Cartesian powers of $\,\bbCP^1\nnh$.

Let us also assume that $\,N\hs$ carries a K\"ah\-ler-Ein\-stein metric 
$\,h\,$ with the Ric\-ci tensor $\,2m\hs h$. In other words, the Ein\-stein 
constant is required to be positive, and its value becomes $\,2m\,$ after 
$\,h\,$ is suitable rescaled.

We denote by $\,M\,$ the projective compactification of the line bundle 
$\,\ey\hs$ over $\,N\,$ arising as the $\,(k/m)\hh$th tensor power of the 
canonical bundle. Our assumptions guarantee the existence in $\,\ey\hs$ of a 
Her\-mit\-i\-an Chern connection with the curvature form equal to the Ric\-ci 
form of $\,h\,$ (see \S\ref{km}) multiplied by $\,-k/m$. The tangent bundle of 
the total space $\,\ey\hs$ is thus decomposed into the direct sum of the 
vertical sub\-bundle $\,\mathcal{V}\nh$, tangent to the fibres, and the 
horizontal distribution $\,\mathcal{H}\,$ of the connection.

The fibre norm in $\,\ey\hs$ is a nonnegative function $\,r\,$ on the total 
space $\,\ey\nh$, which we simultaneously treat as an independent variable. 
Our variable $\,\jt$, restricted to the interval $\,(0,\tc)$, may now be 
turned into a function of the variable $\,r$, characterized by a choice of 
one of the dif\-feo\-mor\-phisms $\,r\mapsto\jt(r)\,$ between the intervals 
$\,(0,\infty)\,$ and $\,(0,\tc)\,$ satisfying the equation
\begin{equation}\label{ddr}
\frac d{dr}\,\jt(r)\,\,=\frac{2\hs\fy(\jt(r))}{kr}\,.
\end{equation}
This allows us to identify $\,\jt\,$ and $\,\fy\,$ with functions 
$\,\ey\smallsetminus N\to\bbR$, depending only on the fibre norm $\,r$, and 
defined on the complement of the zero section $\,N\subset\ey\nh$. Our 
metric $\,g\,$ on $\,\ey\smallsetminus N\hs$ is now defined by requiring that 
$\,\mathcal{V}\,$ be $\,g$-or\-thog\-onal to $\,\mathcal{H}$, that $\,g\,$ 
restricted to each fibre of $\,\ey\hs$ be the Euclidean metric multiplied by 
the  function $\,2(kr)^{-2}e^\jt\fy$, and that $\,g$ restricted to 
$\,\mathcal{H}\,$ be the product of the function $\,2\hh e^\jt$ and the 
pull\-back of the base metric $\,h\,$ under the bundle projection 
$\,\ey\to N\nh$.

A solution $\,\jt(r)\,$ of  (\ref{ddr}) is obviously nonunique; other 
solutions arise from it by rescaling the independent variable $\,r$. Any 
resulting new metric, is, however, isometric to $\,g$, an isometry being 
provided by a suitable rescaling in every fibre of the bundle $\,\ey$.

The rest of the proof -- verifying that $\,g\,$ and $\,\jt\,$ have smooth 
extensions to $\,M\,$ and satisfy the required conditions -- is trivial 
(though tedious).
\end{proof}

\section{Another description of gradient Ric\-ci sol\-i\-tons}\label{io}
Let us restrict our consideration to gradient Ric\-ci sol\-i\-tons with 
{\it positive\/} sol\-i\-ton constants $\,\lc$. A sol\-i\-ton function 
$\,f\,$ satisfying (\ref{ndf}) can thus be {\it normalized\/} by adding a 
suitable constant so as to make the function 
$\,\Delta\ef-g(\nabla\!\ef,\nabla\!\ef)+2\lc\ef\,$ (constant in view 
of Lemma~\ref{stale}) equal to zero. We also normalize the metric $\,g$, 
replacing it with $\,\lambda\hs g\,$ (and in effect assuming that $\,\lc=1$), 
which leaves the Ric\-ci tensor $\,\mathrm{Ric}\,$ and the Hess\-i\-an 
$\,\nabla\df\,$ unchanged. These normalizations result in the equalities
\begin{equation}\label{ndt}
\mathrm{i)}\hskip8pt\nabla\df\,+\,\hh\mathrm{Ric}\hh\,=\,\hs g\hs,\hskip24pt
\mathrm{ii)}\hskip8pt\Delta\ef\,-\,g(\nabla\!\ef,\nabla\!\ef)\,+\,2\ef\,
=\,0\hs.
\end{equation}
For a function $\,\ef\,$ on a Riemannian manifold $\,(M,g)\,$ of dimension 
$\,n\ge3$, the normalized version (\ref{ndt}) of the grad\-i\-ent-sol\-i\-ton 
equation (\ref{ndf}) is equivalent to the following three-equation system:
\[
\begin{array}{rl}
\mathrm{(a)}\hskip4pt&\{\nabla\df+\hh\mathrm{Ric}\}_0^{\phantom j}=0\hs,\\
\mathrm{(b)}\hskip4pt&\Delta\ef-g(\nabla\!\ef,\nabla\!\ef)+2\ef=0\hs,\\
\mathrm{(c)}\hskip4pt&(n-2)(\Delta\ef+\hh\mathrm{s}\hh-n)\hh
+\hh n\hh[\Delta\ef-g(\nabla\!\ef,\nabla\!\ef)+2\ef\hh]\hs=\hs0\hs.
\end{array}
\]
The symbol $\,\{\hskip3.5pt\}_0^{\phantom j}$ denotes here the 
$\,g$-trace\-less part, given by (\ref{trz}), while 
$\,\mathrm{s}=\mathrm{tr}_g\mathrm{Ric}\,$ is the scalar curvature of $\,g$.

The equivalence between (\ref{ndt}) and the system (a) -- (c) is obvious from 
the fact that equality of the two sides in (\ref{ndt}.i) amounts to 
simultaneous equalities of their $\,g$-trace\-less parts, and their 
$\,g$-traces.

\section{Con\-for\-mal changes of Riemannian metrics}\label{zk}
By a {\it con\-for\-mal change\/} of a Riemannian metric $\,g\,$ on a manifold 
$\,M$ one means its replacement by the product $\,\mu g$, where 
$\,\mu:M\to(0,\infty)$.

For a fixed $\,n$-di\-men\-sion\-al Riemannian manifold $\,(M,g)$, let 
$\,\mathrm{Ric}\,$ and $\,\mathrm{s}\,$ denote the Ric\-ci tensor and scalar 
curvature of $\,g$. Their counterparts $\,\br,\bs\,$ for the con\-for\-mal\-ly 
related metric $\,\bg=g/\cf^2\nnh$, with any function $\,\cf:M\to(0,\infty)$, 
are easily verified to be given by
\[
\begin{array}{l}
\br\,=\,\mathrm{Ric}\,+\,(n-2)\hs\cf^{-1}\nabla d\cf\,+\hs
\left[\cf^{-1}\Delta\cf\,-\,(n-1)\cf^{-2}g(\nabla\hskip-1.5pt\cf,\nabla\hskip-1.5pt\cf)\right]\nh g\hs,\hskip-20pt\\
\bs\,\,=\,\,\cf^2\mathrm{s}\hh\,+\,2(n-1)\,\cf\Delta\cf\,
-\,n(n-1)\hh g(\nabla\hskip-1.5pt\cf,\nabla\hskip-1.5pt\cf)\hs.
\end{array}
\]
See, for instance, \cite[pp.\ 528--529]{dillen-verstraelen}.

Suppose now that both $\,\cf\,$ and $\,\ef\,$ are smooth functions of a given 
nonconstant function $\,\jt:M\to\bbR\,$ (cf.\ \S\ref{pr}), and consider 
the functions $\,\jx,\jy\,$ of the variable $\,\jt\,$ defined by 
$\,\jx=2\log\cf-\jt+(m-1)^{-1}\nh\ef\,$ and $\,\jy=(m-1)^{-1}\nh\ef\nh$, where 
$\,m=n/2\,$ (and $\,n\,$ need not be even). Let us further assume that 
$\,g(\nabla\jt,\nh\nabla\jt)\,$ is a smooth function of the function $\,\jt$, 
and set $\,\fy=e^\jt g(\nabla\jt,\nh\nabla\jt)/2$. Thus,
\begin{equation}\label{sgf}
\mathrm{i)}\hskip6pt\cf=e^{(\jx-\jy+\jt)/2}\nh,\hskip11pt\mathrm{ii)}
\hskip6pt\ef=(m-1)\hh\jy\nh,\hskip11pt\mathrm{iii)}
\hskip6ptg(\nabla\jt,\nh\nabla\jt)=2e^{-\jt}\fy\hs.
\end{equation}
Our goal is to find sufficient conditions for a con\-for\-mal change of a 
metric $\,g\,$ (which -- so far -- is completely arbitrary) to yield a 
gradient Ric\-ci sol\-i\-ton $\,\bg\,$ with a sol\-i\-ton constant 
$\,\lc>0\,$ (that is, $\,\lc=1$, after the normalization described in 
\S\ref{io}). These conditions will be imposed on the unknown functions 
$\,\jx,\jy\,$ of the variable $\,\jt$, corresponding -- via (\ref{sgf}) -- to 
the unknown functions $\,\cf\,$ and $\,\ef$, which in turn will constitute 
the functional factor in the con\-for\-mal change $\,\bg=g/\cf^2\nh$, and the 
sol\-i\-to\-n function for $\,\bg$. 

We first evaluate the ingredients of (a) -- (c) in \S\ref{io} for $\,\ef\,$ 
and the metric $\,\bg\,$ (instead of $\,g$), expressing them in terms of: the 
functions $\,\jx,\jy,\fy$ (and their derivatives with respect to $\,\jt$, 
with the notation $\,(\,\,)\dot{\,}=\,d/d\jt$); the $\,g$-Hess\-i\-an 
$\,\nabla d\kp\,$ of the function $\,\kp=e^\jt$; its 
$\,g$-La\-pla\-cian $\,\Delta\kp$; the Ric\-ci tensor $\,\mathrm{Ric}\,$ of 
$\,g\hh$; and its scalar curvature $\,\mathrm{s}$, so that
\begin{equation}\label{nfz}
\{\bn\df\nh+\hh\br\hh\}_0^{\phantom j}
=\{\alpha\nabla d\kp+\hs\mathrm{Ric}\hs
+\beta\,d\jt\otimes d\jt\}_0^{\phantom j}
\end{equation}
where $\,\alpha=(m-1)(\dot\jx+1)\hh e^{-\jt}$ and 
$\,\beta=(m-1)(2\hh\ddot\jx-\dot\jy^2\nh+\dot\jx^2\nh-1)/2$,
\begin{equation}\label{med}
\begin{array}{l}
(m-1)^{-1}e^{\jy-\jx}[\hh\bd\ef\nh-\bg(\bn\!\ef\nh,\bn\!\ef)+2\ef\hh]\\
\hskip20pt=\dot\jy\hh[\hh\Delta\kp\nh-\nh2(\dt+m\hh\fy)]
+2\hh[\fy\hh\ddot\jy\nh-\nh(m-1)\hs\fy\dot\jx\dot\jy\nh+\nh\dt\dot\jy\nh
+\nh\jy\hs e^{\jy-\jx}]\hh,
\end{array}
\end{equation}
and, still with $\,n=2m$,
\begin{equation}\label{eyd}
\begin{array}{l}
e^{\jy-\jx}[(\bd\ef+\hh\bs\hh-2m)\hs
+\hs m(m-1)^{-1}\{\hh\bd\ef\nh-\bg(\bn\!\ef\nh,\bn\!\ef)+2\ef\hh\}]\\
\hskip17pt=\,m\hs[2\hh\dot\jx\dt\hs-\hs(2m-1)\hs\fy\dot\jx^2+\hs\fy\dot\jy^2
+\hs2(\jy-1)\hs e^{\jy-\jx}-\hs\fy\hs+2m]\\
\hskip17pt-\,\,2\hs[m(m-1)\,-\,m(m-1)\fy\,-\,(2m-1)\dt\,-\,\dd\,
-\,e^\jt\mathrm{s}/2]\\
\hskip17pt+\,\,(2m-1)(\dot\jx+1)\hs[\hh\Delta\kp\nh-\nh2(\dt+m\hh\fy)]\\
\hskip17pt-\,\,2\hs[\dd\,-\,(m-1)\hs\dot\jx\dt\,-\,m\hh\fy\,+\,m]\\
\hskip17pt+\,\,(2m-1)\hs[2\hh\ddot\jx\,-\,\dot\jy^2\hs
+\,\dot\jx^2\hs-\,1]\hs\fy\hs.
\end{array}
\end{equation}
Verifying (\ref{nfz}) -- (\ref{eyd}) is, as before, easy, though tedious, and 
may be simplified by using the following intemediate steps. First, due to 
(\ref{sgf}.i),
\begin{equation}\label{tds}
2\hh\dot\cf\cf^{-1}\nh=\hs\dot\jx-\dot\jy+1\hs,\hskip9pt
2(\ddot\cf-\dot\cf)\hh\cf^{-1}\nh=\hs \ddot\jx-\ddot\jy
+[(\dot\jx-\dot\jy)^2\nh-1]/2\hh.
\end{equation}
From the above expression for $\,\br\,$ and $\,\bn\df$, 
(\ref{gns}.ii) -- (\ref{ndh}) with $\,\jh=\cf$ or $\,\jh=\ef$,  
(\ref{sgf}.ii), and (\ref{tds}) with $\,m=n/2\,$ and $\,\kp=e^\jt\nnh$, we 
obtain
\begin{equation}\label{rmr}
\begin{array}{l}
(m-1)^{-1}\{\br-\mathrm{Ric}\}_0^{\phantom j}\\
\hskip48pt=\{(\dot\jx-\dot\jy+1)\hh e^{-\jt}\nabla d\kp
+2(\ddot\cf-\dot\cf)\hh\cf^{-1}\hs d\jt\otimes d\jt\}_0^{\phantom j}\hs,\\
(m-1)^{-1}\bn\df=\dot\jy e^{-\jt}\nabla d\kp\\
\hskip48pt+\,\,(\ddot\jy+\dot\jx\dot\jy-\dot\jy^2)\,d\jt\otimes d\jt
-(\dot\jx-\dot\jy+1)\hh\dot\jy\hh e^{-\jt}\fy\hs g\hh.
\end{array}
\end{equation}
Now (\ref{tds}) and (\ref{rmr}) trivially imply (\ref{nfz}). Applying to the 
second equality in (\ref{rmr}) the operator 
$\,\mathrm{tr}_{\bg}=\cf^2\mathrm{tr}_g$ and, separately, setting 
$\,\jh=\ef\,$ in (\ref{gns}.i), then using the fact that 
$\,\mathrm{tr}_g(d\jt\otimes d\jt)=g(\nabla\jt,\nh\nabla\jt)$, and finally 
expressing $\,\cf\,$ and $\,g(\nabla\jt,\nh\nabla\jt)$ through (\ref{sgf}), 
we get
\begin{equation}\label{moe}
\begin{array}{l}
(m-1)^{-1}e^{\jy-\jx}\bd\ef\,=\,\dot\jy\Delta\kp\,
+\,2\hh[\ddot\jy-m\dot\jy+(m-1)(\dot\jy-\dot\jx)\dot\jy]\hh\fy\hs,\\
(m-1)^{-2}e^{\jy-\jx}\bg(\bn\!\ef\nh,\bn\!\ef)=2\hh\dot\jy^2\fy\hs,
\end{array}
\end{equation}
which gives (\ref{med}). From the formula for $\,\bs$, (\ref{gns}) -- 
(\ref{ndh}) with $\,\jh=\cf$, and (\ref{tds}), we have
\begin{equation}\label{eys}
\begin{array}{l}
(2m-1)^{-1}e^{\jy-\jx}\hs\bs\,=\,(2m-1)^{-1}e^\jt\mathrm{s}
+(\dot\jx-\dot\jy+1)\hh\Delta\kp\\
\hskip22pt+\,\,[2\hh\ddot\jx-2\hh\ddot\jy-(m-1)(\dot\jx-\dot\jy)^2\nh
-2m(\dot\jx-\dot\jy)-m-1]\hs\fy\hs.
\end{array}
\end{equation}
The last equality, (\ref{moe}) and (\ref{med}) easily yield (\ref{eyd}).

\section{Ric\-ci sol\-i\-tons con\-for\-mal to K\"ah\-ler metrics}\label{kk}
Having fixed $\,\tc\in(0,\infty)\,$ and integers $\,m,k\,$ such that 
$\,m>k>0$, let us consider the following system of sec\-ond-or\-der 
differential equations:
\begin{equation}\label{xdd}
\begin{array}{rrl}
\mathrm{i)}\hskip4pt&
2\hh\ddot\jx\,\,=&\dot\jy^2\hs\,-\,\,\dot\jx^2\hs\,+\,\,1\hs,\\
\mathrm{ii)}\hskip4pt&
\fy\hh\ddot\jy\,\,=&(m-1)\hs\fy\dot\jx\dot\jy\,\,-\,\,\dt\dot\jy\,\,
-\,\,\jy e^{\jy-\jx}\nh,\\
\mathrm{iii)}\hskip4pt&
\dd\,\,=&(m-1)\hs\dot\jx\dt\,\,+\,\,m\hh\fy\,\,-\,\,m\hs,
\end{array}
\end{equation}
imposed on unknown $\hs C^\infty\nnh$ functions 
$\,\jx,\jy,\fy:[\hs0,\tc\hh]\to\bbR\,$ of 
the variable $\,\jt\in[\hs0,\tc\hh]$, where $\,(\,\,)\dot{\,}=\,d/d\jt$, and 
the additional first-or\-der equation 
\begin{equation}\label{int}
2\hh\dot\jx\dt\,-\,(2m-1)\hs\fy\dot\jx^2\hs+\,\fy\dot\jy^2\hs
+\,2(\jy-1)\hs e^{\jy-\jx}\hs-\,\fy\,+2m\hs\,=\,\hs0
\end{equation}
along with the boundary conditions
\begin{equation}\label{bdr}
\fy(0)=\fy(\tc)=0\hh,\hskip7pt\fy>0\hskip5pt\mathrm{on}\hskip5pt(0,\tc)\hh,
\hskip7pt\dt(0)=k=-\hh\dt(\tc)\hh.
\end{equation}
Equation (\ref{int}) does not reduce too drastically the solution set of 
(\ref{xdd}), since the left-hand side of (\ref{int}) multiplied by $\,e^\jx$ 
is an integral of (\ref{xdd}).
\begin{theorem}\label{ckcrs}With any solution\/ 
$\,(\jx,\jy,\fy):[\hs0,\tc\hh]\to\rtr$ of the system\/ {\rm(\ref{xdd})} -- 
{\rm(\ref{bdr})}, for\/ $\,\tc,m,k\,$ fixed as above, one can associate a  
Ric\-ci sol\-i\-ton\/ $\,\bg\,$ con\-for\-mal to a K\"ah\-ler metric on the 
compact complex manifold\/ $\,M=M^m_k$ defined by\/ $(*)$ in\/ 
\S{\rm\ref{sk}}. To obtain\/ $\,\bg$, we use the K\"ah\-ler metric\/ $\,g\,$ 
on\/ $\,M\,$ and the function\/ $\,\jt:M\to\bbR$, arising from our\/ 
$\,\jx,\fy,\tc,m,k$ via Lemma\/~{\rm\ref{rhxpl}}, setting\/ $\,\bg=g/\cf^2$ 
for the function\/ $\,\cf\,$ in\/ {\rm(\ref{sgf}.i)}.

The metric\/ $\,\bg\,$ and the function\/ $\,\ef\,$ defined by\/ 
{\rm(\ref{sgf}.ii)} then satisfy conditions\/ {\rm(a)} -- {\rm(c)} of\/ 
\S{\rm\ref{io}}, equivalent to the grad\-i\-ent-sol\-i\-ton  equation\/ 
{\rm(\ref{ndf})} for\/ $\,\lambda=1\,$ and for\/ $\,\bg\,$ instead of\/ 
$\,g$.
\end{theorem}
\begin{proof}[Proof]The right-hand sides in equations (\ref{nfz}) -- 
(\ref{eyd}) are all equal to zero -- in (\ref{nfz}) we have 
$\,\{\alpha\nabla d\kp+\hs\mathrm{Ric}\}_0^{\phantom j}\nh=0\,$ and 
$\,\beta=0$, by Lemma~\ref{rhxpl}(i) and (\ref{xdd}.i); in (\ref{med}), both 
terms on the right-hand side vanish as a consequence of Lemma~\ref{rhxpl}(ii) 
and (\ref{xdd}.ii); while in (\ref{eyd}) each of the final five lines is zero 
due to (\ref{int}), Lemma~\ref{rhxpl}(iii), Lemma~\ref{rhxpl}(ii), 
(\ref{xdd}.iii) and (\ref{xdd}.i).
\end{proof}
For all known examples of even-di\-men\-sion\-al compact Ric\-ci sol\-i\-tons 
(listed in Question~\ref{isthr}), the multiplicities of eigen\-val\-ues of the 
Ric\-ci tensor are even at every point. This follows from (\ref{grh}) and the 
fact that the Ric\-ci tensor of a Riemannian product is, in a natural sense, 
the direct sum of the Ric\-ci tensors of the factor manifolds.

Suppose that, for some solution $\,(\jx,\jy,\fy):[\hs0,\tc\hh]\to\rtr$ of 
the system (\ref{xdd}) -- (\ref{bdr}), and the function $\,\cf\,$ defined by 
(\ref{sgf}.i), one has $\,\ddot\cf\ne\dot\cf$ somewhere in the interval 
$\,[\hs0,\tc\hh]\,$ (whether such solutions exists, is not known; the issue is 
further complicated by the singularities of equation (\ref{xdd}.ii) at 
$\,\jt=0\,$ and $\,\jt=\tc$, due to (\ref{bdr})). For reasons named below, the 
compact Ric\-ci sol\-i\-ton $\,(M,\bg)\,$ arising then from 
Theorem~\ref{ckcrs} would be non-Ein\-stein, non-K\"ah\-ler (even locally), 
and not locally isometric to a Riemannian product with Ein\-stein or 
lo\-cal\-ly-K\"ah\-ler factor manifolds. The result would be an affirmative 
answer to Question~\ref{isthr}.

The reasons just mentioned are as follows. Since the real dimension of $\,M\,$ 
is even, at every point satisfying the condition 
$\,(\ddot\cf-\dot\cf)\,d\jt\ne0$ the Ric\-ci tensor $\,\br\,$ has an  
eigen\-val\-ue of odd multiplicity, for which the gradient $\,\bn\kp\,$ is 
an eigen\-vec\-tor. This is immediate from (\ref{rmr}) since, in view of 
Lemma~\ref{rhxpl}(i) and (\ref{grh}), $\,\nabla d\kp\,$ and $\,\mathrm{Ric}\,$ 
are Her\-mit\-i\-an and simultaneously di\-ag\-o\-nal\-izable at each point, 
with even-di\-men\-sion\-al eigen\-spaces, while from (\ref{toz}.c) and the 
second equality in Lemma~\ref{rhxpl}(ii) it follows that $\,\nabla\kp\,$ is 
a common  eigen\-vec\-tor of both $\,\nabla d\kp\,$ and $\,\mathrm{Ric}$.

Solutions of the system (\ref{xdd}) -- (\ref{int}) for which 
$\,\ddot\cf=\dot\cf\,$ are in turn completely understood. We will describe 
them in \S\S\ref{sy}--\ref{pk}.

Theorem~\ref{ckcrs} also has a {\it local version}, in which instead of 
the bound\-a\-ry conditions (\ref{bdr}) one only assumes positivity of 
$\,\fy\,$ on the open interval forming the domain of the solution 
$\,(\jx,\jy,\fy)\,$ to (\ref{xdd}) -- (\ref{int}) with $\,m\ge2$. The 
construction of Lemma~\ref{rhxpl} then gives a metric $\,g\,$ and function 
$\,\jt$ which, although defined just on some noncompact manifold, still 
satisfy the conclusions (i) -- (iii) in Lemma~\ref{rhxpl}.

Yet another generalization of Theorem~\ref{ckcrs} arises when one removes 
equation (\ref{xdd}.ii) from the system (\ref{xdd}) -- (\ref{bdr}), keeping 
all the remaining requirements. The right-hand side of (\ref{nfz}) will then 
still be equal to zero. This leads to a weaker version of the 
grad\-i\-ent-sol\-i\-ton equation (\ref{ndf}): 
$\,\bn\df\nh+\hh\br=\lambda\hs\bg\,$ for some function 
$\,\lambda\,$ that need not be constant. Metrics $\,\bg\,$ for which such 
functions $\,f\,$ and $\,\lambda\,$ exist were studied by several authors 
\cite[p.\ 369]{maschler}, \cite{barros-ribeiro}, 
\cite{pigola-rigoli-rimoldi-setti}; they are sometimes called (gradient) 
{\it Ric\-ci al\-most-sol\-i\-tons}.

\section{Symmetries of {\rm(\ref{xdd})} -- {\rm(\ref{int})} and the 
condition $\,\ddot\cf=\dot\cf$}\label{sy}
In the equations forming the system (\ref{xdd}) -- (\ref{int}), no term 
depends explicitly on the variable $\,\jt$, while the terms containing the 
first derivatives $\,\dot\jx,\dot\jy,\dt\,$ are homogeneous quadratic in them. 
The system will thus remain satisfied if in a solution, defined on any 
interval, one replaces the variable $\,\jt\,$ by $\,c\pm\jt$, with a constant 
$\,c$.

The set of solutions defined on the interval $\,[\hs0,\tc\hh]\,$ and 
satisfying the boundary conditions (\ref{bdr}) is therefore invariant under 
the substitution of $\,\tc\nnh-\hs\jt\,$ for the variable $\,\jt$.

We now discuss the solutions of (\ref{xdd}) -- (\ref{int}) such that 
$\,\ddot\cf=\dot\cf$.
\begin{lemma}\label{snses}Let a solution\/ $\,(\jx,\jy,\fy)\,$ of\/ 
{\rm(\ref{xdd})} -- {\rm(\ref{int})}, with\/ $\,m\ge2$, defined on an open 
interval of the variable\/ $\,\jt$, satisfy in addition the condition\/ 
$\,\ddot\cf=\dot\cf$, where the function\/ $\,\cf\,$ is given by\/ 
{\rm(\ref{sgf}.i)}. Then\/ $\,\cf=q_0^{\phantom j}+q_1^{\phantom j}e^\jt$ for 
some constants\/ $\,q_0^{\phantom j},\hs q_1^{\phantom j}$, at least one of 
which is positive. Furthermore,
\begin{equation}\label{eta}
\mathrm{i)}\hskip8pt\dot\jy\hs\dt\,=\,(m\dot\jx-\dot\jy)\hh\dot\jy\hs\fy\,
-\,\jy\hs e^{\jy-\jx}\hs,\hskip28pt\mathrm{ii)}\hskip8pt\ddot\jy\,
=\,(\dot\jy-\dot\jx)\hh\dot\jy\hs,
\end{equation}
and one of the following three cases must occur\/{\rm:}
\begin{alternative}
\item $\jy=0\,$ on the entire interval of the variable\/ $\,\jt$,
\item $q_1^{\phantom j}=0$, so that\/ $\,\cf=q_0^{\phantom j}$ is a positive 
constant,
\item $q_0^{\phantom j}=0<q_1^{\phantom j\,}\,$ and\/ 
$\,\,\cf=q_1^{\phantom j}e^\jt\nnh$.
\end{alternative}
In the solution set of the full system\/ {\rm(\ref{xdd})} -- {\rm(\ref{bdr})}, 
with\/ $\,\tc,m,k\,$ fixed as before, the variable substitution of\/ 
$\,\tc\nnh-\hs\jt\,$ for\/ $\,\jt\,$ preserves the condition\/ 
$\,\ddot\cf=\dot\cf$, leaves case\/ {\rm(i)} unchanged, and switches case\/ 
{\rm(ii)} with\/ {\rm(iii)}.
\end{lemma}
\begin{proof}[Proof]Let us fix a solution $\,(\jx,\jy,\fy)\,$ of (\ref{xdd}). 
For the function $\,\eta\,$ defined to be the difference between the left-hand 
and right-hand sides of (\ref{eta}.i), subtracting (\ref{xdd}.i) multiplied by 
$\,\fy/2\,$ from (\ref{xdd}.ii) we see, using (\ref{tds}.b) and (\ref{xdd}), 
that $\,\eta=2(\ddot\cf-\dot\cf)\hh\fy\hh\cf^{-1}\nnh$. Also, by (\ref{xdd}), 
the expression
\[
2\hh(\fy\eta)\dot{\,}\nh-2(m-1)\hs\fy\eta\dot\jx
+4(m-1)(\dot\cf-\cf)\hs\fy^2\cf^{-2}\dot\jy\hh\dot\cf
\]
is the product of $\,-\hs\fy\dot\jy\,$ and the left-hand side of 
(\ref{int}), as (\ref{tds}.a) gives 
$\,4(\dot\cf-\cf)\hh\cf^{-2}\dot\cf=(2\hh\dot\cf\cf^{-1})^2\nh
-2(2\hh\dot\cf\cf^{-1})=(\dot\jx-\dot\jy+1)(\dot\jx-\dot\jy-1)$. From 
(\ref{xdd}) -- (\ref{int}) with $\,\ddot\cf=\dot\cf\,$ we thus get 
(\ref{eta}.i) (that is, vanishing of $\,\eta$) and
\begin{equation}\label{dsd}
(\dot\cf-\cf)\hs\fy\dot\jy\hh\dot\cf\,=\,0\hs.
\end{equation}
By (\ref{eta}.i) and (\ref{xdd}.ii), 
$\,(m\dot\jx\dot\jy-\dot\jy^2)\hs\fy=\dt\dot\jy+\jy\hs e^{\jy-\jx}
=[(m-1)\dot\jx\dot\jy-\ddot\jy]\hs\fy$ while, from (\ref{xdd}.iii), 
$\,\fy\ne0\,$ on some dense subset of the domain interval. This proves 
(\ref{eta}.ii).

On the other hand, if $\,\ddot\cf=\dot\cf$, then 
$\,\dot\cf=q_1^{\phantom j}e^\jt$ and 
$\,\cf=q_0^{\phantom j}+q_1^{\phantom j}e^\jt$ with constants 
$\,q_0^{\phantom j},\hs q_1^{\phantom j}$, which cannot be both nonpositive, 
since $\,\cf>0\,$ due to (\ref{sgf}.i). Under the hypotheses of the lemma, 
equation (\ref{dsd}) gives rise (in view of analyticity of the solution 
$\,(\jx,\jy,\fy)\,$ of (\ref{xdd}) on every connected component of the dence 
set on which $\,\fy\ne0$) to four possible cases: $\,\fy=0$, $\,\dot\jy=0$, 
$\,\dot\cf=0\,$ and $\,\dot\cf=\cf$. The first one is excluded by 
(\ref{xdd}.iii); the second, as a consequence of (\ref{xdd}.ii), amounts to 
requiring that $\,\jy=0\hs$; the third case gives (ii); the fourth -- (iii). 

If (\ref{xdd}) -- (\ref{bdr}) are assumed, the substitution of 
$\,\tc\nnh-\hs\jt\,$ for $\,\jt\,$ obviously preserves condition (i) while, by 
(\ref{sgf}.i), it causes $\,\cf\,$ to be replaced with the function 
$\,\jt\mapsto e^{\jt\hs-\hs\tc/2}\cf(\tc\nnh-\hs\jt)$, which completes the 
proof.
\end{proof}

\section{Case (i) in Lemma~\ref{snses}}\label{pp}
Let us consider the case $\,\jy=0\,$ in Lemma~\ref{snses}. In view of 
(\ref{sgf}.ii), this is nothing else than vanishing of the normalized 
sol\-i\-ton function $\,\ef$ in Theorem~\ref{ckcrs} (or -- more precisely 
-- in its local version; see \S\ref{kk}). Our construction thus leads now to 
Ein\-stein metrics $\,\bg$, cf.\ \S\ref{me}.

We describe below the solutions $\,(\jx,\jy,\fy)\,$ of the system formed by 
(\ref{xdd}) -- (\ref{int}) and equation $\,\jy=0\,$ on any open interval; 
translating the variable $\,\jt\,$ (as in \S\ref{sy}), we may assume that the 
interval contains $\,0$.
\begin{theorem}\label{ruyez}For all solutions $\,(\jx,\jy,\fy)\,$ of\/ 
{\rm(\ref{xdd})} -- \,{\rm(\ref{int})} such that\/ $\,\jy=0$, on an open 
interval containing\/ $\,0$, the initial data
\begin{equation}\label{ini}
(\jx_0^{\phantom j},\,\jy_0^{\phantom j},\,\fy{}_0^{\phantom j},\,
\dot\jx_0^{\phantom j},\,\dot\jy_0^{\phantom j},\,\dt{}_0^{\phantom j})\,
=\,(x(0),\,y(0),\,\fy(0),\,\dot\jx(0),\,\dot\jy(0),\,\dt(0))
\end{equation}
satisfy the conditions
\begin{equation}\label{war}
\jy_0^{\phantom j}\nh=\dot\jy_0^{\phantom j}\nh=0\hh,\hskip11pt
2\hh\dot\jx_0^{\phantom j}\dt{}_0^{\phantom j}\nh
=[(2m-1)\dot\jx_0^2\nh+1\hs]\hs\fy{}_0^{\phantom j}\nh
+2\hh e^{-\jx_0^{\phantom j}}\nh-2m\hs.
\end{equation}
Conversely, every choice of the data\/ {\rm(\ref{ini})} satisfying 
{\rm(\ref{war})} is realized by a unique solution\/ $\,(\jx,\jy,\fy)\,$ of\/ 
{\rm(\ref{xdd})} -- \,{\rm(\ref{int})} with\/ $\,\jy=0$, defined on a maximal 
interval containing\/ $\,0$. One then has
\begin{equation}\label{xez}
\jx=\jx_0^{\phantom j}\nh+2\hs\log\hh|\Theta|\hs,\hskip10pt\mathrm{where}
\hskip8pt\Theta=\cosh\,(\jt/2)+\dot\jx_0^{\phantom j}\hh\sinh\,(\jt/2)\hs,
\end{equation}
$\,\jy=0$, and\/ $\,\fy\,$ is the unique solution of the first-or\-der 
linear equation 
\begin{equation}\label{txf}
2\hh\dot\jx\dt=[(2m-1)\dot\jx^2\nh+1\hs]\hs\fy+2\hh e^{-\jx}\nh-2m
\end{equation}
with the initial conditions\/ $\,\fy(0)=\fy{}_0^{\phantom j}$, 
$\,\dt(0)=\dt{}_0^{\phantom j}$.

The linear equation\/ {\rm(\ref{txf})} has an obvious integrating factor, 
which allows us to rewrite it, for\/ $\,\jt\hn\,$ with\/ 
$\,\Theta(t)\dot\Theta(t)\ne0$, as\/ $\,(G\hh\fy)\dot{\,}\hs=\hs F\nh$, where
\begin{equation}\label{get}
G\,=\,2(\Theta^{2m-1}\dot\Theta)^{-1},\hskip22ptF\,
=\,(\Theta^m\dot\Theta)^{-2}(e^{-\jx_0^{\phantom j}}\nh-m\hh\Theta^2)\hh.
\end{equation}
\end{theorem}
\begin{proof}[Proof]Our system consists of (\ref{xdd}), (\ref{int}) and the 
condition $\,\jy=0\,$ (which, due to (\ref{xdd}.i) and (\ref{tds}.b), imply 
the equality $\,\ddot\cf=\dot\cf$). In other words, we have two unknown 
functions, $\,\jx\,$ and $\,\fy$, subject to just two equations: 
$\,2\hh\ddot\jx=1-\dot\jx^2$ and (\ref{txf}); note that (\ref{xdd}.iii) 
follows from them. Furthermore, the existence and uniqueness of a solution 
$\,\fy\,$ to (\ref{txf}) with the stated initial conditions are obvious; this 
is so even in the singular case, that is, when 
$\,\dot\jx_0^{\phantom j}\nh=0$, as one easily verifies using the integrating 
factor $\,[\hs\cosh(\jt/2)]^{-2m}\coth(\jt/2)$.
\end{proof}

\newpage
\section{Page's and B\'erard Bergery's examples}\label{pb}
Having fixed $\,m,k\in\bbZ\,$ such that $\,m>k>0\,$ and $\,a\in(-1,0)$, we 
define the data (\ref{ini}) with (\ref{war}), functions $\,\Theta,G,F$, and a 
constant $\,\tc>0$, by 
$\,\jx_0^{\phantom j}\nh=-\hs\log\hh(ka+m)\hh,\,\,\,\,
\jy_0^{\phantom j}\nh=\fy{}_0^{\phantom j}\nh=\dot\jy_0^{\phantom j}\nh=0\hh,
\,\,\,\,\dot\jx_0^{\phantom j}\nh=a\hh,\,\,\,\,\dt{}_0^{\phantom j}\nh=k\hh,$ 
formulae (\ref{xez}) and (\ref{get}), and $\,\tc=2\hh\log\hh[(1-a)/(1+a)]$. 
Thus
\begin{equation}\label{ttt}
\begin{array}{rrcl}
\mathrm{i)}&
2\hh\Theta(\jt)&=&(1+a)\hs e^{\jt/2}\hs+\,(1-a)\hs e^{-\jt/2}\hs>\,0\hs,\\
\mathrm{ii)}&
4\hh\dot\Theta(\jt)&=&(1+a)\hs e^{\jt/2}\hs-\,(1-a)\hs e^{-\jt/2}\hs,\\
\mathrm{iii)}&
\Theta^2\nh-4\hh\dot\Theta^2&=&1\hs-\hs a^2\hs>\,0\hs,\hskip22pt
4\hh\ddot\Theta\,=\,\Theta\,>\,0\hs,\\
\mathrm{iv)}&
F&=&[ka+m-m\hh\Theta^2\hh]\hs(\Theta^m\dot\Theta)^{-2}\nh.
\end{array}
\end{equation}
For the polynomial $\,S\,$ in the variables $\,a,\xi\,$ given by
\[
S(a,\xi)\,=\,\displaystyle{\sum_{j=0}^m\frac{(-1)^j\xi^{2j}}{2j-1}}
[\textstyle{{m-1\choose j}}\hs(ma+k)a
+\textstyle{{m-1\choose j-1}}\hs(ka+m)]\hs,
\]
where $\,\textstyle{{m-1\choose j}}=0\,$ if $\,j<0\,$ or $\,j\ge m$, the 
expression $\,P(a)=a^{-1}S(a,a)$ is a polynomial in the variable $\,a$, for 
which
\begin{equation}\label{sko}
\begin{array}{l}
\ \\
P(0)\,=\,-k\,<\,0\hs,\\
P(-k/m)\,=\,(1-k/m)(1+m/k)\displaystyle{\int_0^{\hskip2ptk/m}}
(1-a^2)^{m-1}\hskip2ptda\,>\,0\hs.
\end{array}
\end{equation}
The first equality in (\ref{sko}) is obvious due to the definition of $\,P$, 
while the second one is easily obtained from our formula for $\,S(a,\xi)\,$ by 
using a binomial expansion of the integrand and the fact that the factor 
$\,ma+k$ in $\,S(a,\xi)\,$ vanishes when $\,a=-k/m$.

As a consequence of (\ref{sko}), we may now fix $\,a\in(-k/m,0)\,$ such that 
$\,P(a)=0$, that is, $\,S(a,a)=0$. Since $\,S(a,\xi)\,$ is an even function of 
$\,\xi$, we also get $\,S(a,-\hh a)=0$.

Replace $\,\jt\in[\hs0,\infty)\,$ by the new variable 
$\,\xi=2\hh\dot\Theta(\jt)/\Theta(\jt)\in[\hh a,1)$. From (\ref{ttt}.i), 
(\ref{ttt}.iii) and (\ref{ttt}.i-ii) it follows, respectively, that this makes 
sense, and that $\,|\hh\xi|<1\,$ and $\,2\hs\dot\xi=1-\xi^2\nh>0$, while 
$\,\xi\to1\,$ as $\,\jt\to\infty$. Thus, $\,\jt\mapsto\xi\,$ is a 
dif\-feo\-mor\-phism of the interval $\,[\hs0,\infty)\,$ onto 
$\,[\hh a,1)$, easily verified -- if one uses (\ref{ttt}.i-ii) again -- to 
send $\,\jt=0\,$ and $\,\jt=\tc$, for our positive constant 
$\,\tc=2\hh\log\hh[(1-a)/(1+a)]$, to $\,\xi=a\,$ and, respectively, 
$\,\xi=-\hs a$. The first part of (\ref{ttt}.iii) gives 
$\,\xi^2\nh=1+(a^2\nh-1)\Theta^{-2}\nnh$, that is, 
$\,\Theta^2\nh=(1-a^2)/(1-\xi^2)$, and 
$\,4\hh\dot\Theta^2\nh=(1-a^2)\xi^2/(1-\xi^2)\,$ (since 
$\,4\hh\dot\Theta^2/\Theta^2\nh=\xi^2$). Expressing $\,G\,$ and $\,F\,$ 
through $\,\xi$, we obtain, from (\ref{get}),
\begin{equation}\label{gef}
G\,=\,\frac{4(1-\xi^2)^m}{(1-a^2)^m\xi}\,,
\end{equation}
as $\,\Theta^{2m-1}\dot\Theta=\Theta^{2m}\dot\Theta/\Theta
=(\Theta^2){}^m\xi/2$, and, in view of (\ref{ttt}.iv),
\begin{equation}\label{fef}
F\,=\,4(1-a^2)^{-m-1}
[(ma+k)a\hs\xi^{-2}\nh-(ka+m)]\hs(1-\xi^2)^m\hh.
\end{equation}
The above formula for $\,S(a,\xi)\,$ and (\ref{fef}) easily show that
\begin{equation}\label{dsx}
4\hh(1-\xi^2)\,d\hs[S(a,\xi)/\xi\hh]\hh/\nh d\hh\xi\,\,
=\,\,(1-a^2)^{m+1}F\hh,
\end{equation}
The relation $\,F=(G\hh\fy)\dot{\,}\,$ in Theorem~\ref{ruyez} amounts to the 
equality $\,2F=(1-\xi^2)\hs d\hs(G\hh\fy)/d\hs\xi$. It is therefore satisfied 
by $\,\fy\,$ such that
\begin{equation}\label{fes}
(1-a^2)(1-\xi^2)\hh\fy\,=\,2S(a,\xi)\hs.
\end{equation}
Consequently, our choice of $\,a\,$ causes this function $\,\fy\,$ of the 
variable $\,\xi$ to vanish for $\,\xi=\pm\hs a$. In other words, treating 
$\,\fy\,$ as a function of $\,\jt$, we have $\,\fy(0)=\fy(\tc)=0$, where 
$\,\tc=2\hh\log\hh[(1-a)/(1+a)]$.

Next, $\,\fy\,$ given by (\ref{fes}) also satisfies the remaining boundary 
conditions (\ref{bdr}). Namely, the definition of $\,S(a,\xi)\,$ and 
(\ref{gef}) -- (\ref{fef}) yield
\[
(1-a^2)^{m+1}(1-\xi^2)^{-m}\xi^2(F\pm kG)\hs
=\hs-\hh4\hh(\xi\pm a)\hs[(ka+m)\hh\xi\mp(ma+k)]\hh,
\]
and so $\,F/G=\pm k\,$ for $\,\xi=\pm a$. Switching to the variable $\,\jt$, 
we obtain $\,F(0)/G(0)=k\,$ and $\,F(\tc)/G(\tc)=-\hh k$. The equalities 
$\,\fy(0)=\fy(\tc)=0$ and $\,F=(G\hh\fy)\dot{\,}\,$ now show that 
$\,\dt(0)=k\,$ and $\,\dt(\tc)=-k$.

Positivity of $\,\fy\,$ on $\,(0,\tc)\,$ is thus reduced to nonvanishing of 
$\,\fy\,$ when $\,\xi\in(a,-\hs a)$. If, however, $\,\fy\,$ vanished for some 
$\,\xi\in(a,-\hs a)$, (\ref{fes}) would -- first -- yield $\,\xi\ne0\,$ (as 
the definition of $\,S(a,\xi)\,$ and the inequalities $\,-k/m<a<0\,$ give 
$\,S(a,0)=-(ma+k)a>0$) and -- secondly -- allow us to assume that 
$\,\xi\in(a,0)\,$ (since $\,S(a,\xi)\,$ is an even function of $\,\xi$). 
Vanishing of $\,\fy$, and consequently of $\,S(a,\xi)/\xi$, for both this 
value of $\,\xi$ and for $\,\xi=a$, would imply vanishing of the derivative 
$\,d\hs[S(a,\xi)/\xi\hh]\hh/\nh d\hh\xi$ somewhere in $\,[a,0)\,$ which, 
combined with (\ref{dsx}), leads to a contradiction: by (\ref{fef}), $\,F<0\,$ 
when $\,a\in(-k/m,0)\,$ and $\,0<|\hh\xi|<1$.

Theorem~\ref{ckcrs} now implies that the solution $\,(\jx,\jy,\fy)\,$ of 
(\ref{xdd}) -- (\ref{int}), corresponding to the above data (\ref{ini}), with 
our $\,a\in(-k/m,0)$, allows us to construct a K\"ah\-ler-Ric\-ci sol\-i\-ton 
on each of the compact complex manifolds $\,M^m_k$. We thus obtain the {\it 
examples found by Page'a\/} (for $\,m=2$) {\it and by B\'erard Bergery\/} 
(if $\,m>2$).

The above constructions can also be found in Besse's book 
\cite[pp.\ 273\hs--275]{besse} and, obviously, the original papers 
\cite{page,berard-bergery}.

\section{Case (ii) in Lemma~\ref{snses}}\label{pd}
Constancy of $\,\cf\,$ in Lemma~\ref{snses} implies that $\,\bg=g/\cf^2$ is a 
K\"ah\-ler metric, since -- in the local version of Theorem~\ref{ckcrs} -- it 
arises from a {\it trivial\/} con\-for\-mal change of the K\"ah\-ler metric 
$\,g$.

In the following description of solutions $\,(\jx,\jy,\fy)\,$ to (\ref{xdd}) 
-- (\ref{int}) with a constant function $\,\cf\,$ we are assuming, without 
loss of generality (cf.\ \S\ref{pk}) that the domain interval contains $\,0$.
\begin{theorem}\label{rusjs}If a solution\/ $\,(\jx,\jy,\fy)\,$ of\/ 
{\rm(\ref{xdd})} -- \,{\rm(\ref{int})} defined on an open interval 
containing\/ $\,0\,$ has the property that the function\/ $\,\cf\,$ is 
constant, then the initial data\/ {\rm(\ref{ini})} must satisfy the conditions
\begin{equation}\label{cnd}
\begin{array}{l}
\dot\jy_0^{\phantom j}\nh-\dot\jx_0^{\phantom j}\nh=1\hh,\hskip24pt
(1-m\hs e^{\jx_0^{\phantom j}\nh-\jy_0^{\phantom j}})\hs
\dot\jy_0^{\phantom j}\nh=\jy_0^{\phantom j}\hh,\\
\dt{}_0^{\phantom j}\nh=[(m-1)\dot\jy_0^{\phantom j}
-m\hh]\hs\fy{}_0^{\phantom j}+m-
e^{\jy_0^{\phantom j}\nh-\jx_0^{\phantom j}}\hh.
\end{array}
\end{equation}
Conversely, any data\/ {\rm(\ref{ini})} with\/ {\rm(\ref{cnd})} are realized 
by a unique solution\/ $\,(\jx,\jy,\fy)\,$ of\/ {\rm(\ref{xdd})} -- 
\,{\rm(\ref{int})} with a constant function\/ $\,\cf$, defined on the whole 
real line. Explicitly,
\begin{equation}\label{xex}
\jx=\jx_0^{\phantom j}\nh-\jt+\dot\jy_0^{\phantom j}(e^\jt\nh-1)\hh,\hskip22pt
\jy=\jy_0^{\phantom j}\nh+\dot\jy_0^{\phantom j}(e^\jt\nh-1)\hh, 
\end{equation}
and\/ $\,\fy\,$ is the unique solution of the first-or\-der linear equation 
\begin{equation}\label{dfe}
\dt=[(m-1)\dot\jy_0^{\phantom j}e^\jt\nh-m\hh]\hs\fy\hs+m
-e^{\jy_0^{\phantom j}\nh-\jx_0^{\phantom j}\nh+\hs\jt}
\end{equation}
with the  initial conditions\/ $\,\fy(0)=\fy{}_0^{\phantom j}$, 
$\,\dt(0)=\dt{}_0^{\phantom j}$.

Equation\/ {\rm(\ref{dfe})} may also be expressed as\/ 
$\,(G\hh\fy)\dot{\,}\hs=\hs F\nh$, where
\begin{equation}\label{gte}
G(\jt)=\exp\,[\hh m\hh\jt\hh-\hh(m-1)\hh\dot\jy_0^{\phantom j}e^\jt\hh]\hh,\hskip7pt
F(\jt)\hs=(m-e^{\jy_0^{\phantom j}\nh-\jx_0^{\phantom j}\nh+\hs\jt})
\hskip1.2ptG(\jt)\hh.
\end{equation}
\end{theorem}
\begin{proof}[Proof]Constancy of $\,\cf\,$ and (\ref{tds}.a) give 
$\,\dot\jy-\dot\jx=1$, so that, by (\ref{eta}.ii), $\,\ddot\jy=\dot\jy$, which 
yields (\ref{xex}) and the first equality in (\ref{cnd}).

To prove (\ref{dfe}), consider two possible cases. In the first one, 
$\,\dot\jy_0^{\phantom j}\nh=0$. From (\ref{xex}) we thus obtain constancy of 
$\,\jy\,$ and the equality $\,\jx=\jx_0^{\phantom j}\nh-\jt$. Using 
(\ref{xdd}.ii) we see that $\,\jy=0$. Conclusion (\ref{txf}) in 
Theorem~\ref{ruyez} for $\,\jx=\jx_0^{\phantom j}\nh-\jt\,$ now implies 
(\ref{dfe}) for $\,\jy_0^{\phantom j}\nh=\dot\jy_0^{\phantom j}\nh=0$.

In the remaining case, $\,\dot\jy_0^{\phantom j}\nh\ne0$. Dividing 
(\ref{xdd}.ii) by $\,\dot\jy=\ddot\jy=\dot\jy_0^{\phantom j}e^\jt\nh\ne0$, and 
then replacing $\,\dot\jx\,$ (or, $\,\jy\hs e^{\jy-\jx}$) with 
$\,\dot\jy_0^{\phantom j}e^\jt\nh-1\,$ (or, respectively, 
$\,[\jy_0^{\phantom j}\nh+\dot\jy_0^{\phantom j}(e^\jt\nh-1)]\hs 
e^{\jy_0^{\phantom j}\nh-\jx_0^{\phantom j}\nh+\hs\jt}$), we obtain
\begin{equation}\label{dfm}
\dt\,=\,[(m-1)\dot\jy_0^{\phantom j}e^\jt\nh-m\hh]\hs\fy\,
-\,e^{\jy_0^{\phantom j}\nh-\jx_0^{\phantom j}}(e^\jt\nh-1
+\jy_0^{\phantom j}/\dot\jy_0^{\phantom j})\hh.
\end{equation}
Let us use (\ref{dfm}) and the equation obtained by differentiating 
(\ref{dfm}) to express $\,\dt\,$ and $\,\dd\,$ in terms of 
$\,\jx_0^{\phantom j},\jy_0^{\phantom j},\dot\jy_0^{\phantom j}$ and $\,\fy$. 
The resulting expressions and the equality 
$\,\jx=\dot\jy_0^{\phantom j}e^\jt\nh-1\,$ allow us to rewrite the difference 
between the two sides of (\ref{xdd}.iii), so as to obtain both the second 
equality of (\ref{cnd}), and (\ref{dfe}). We also see that the second equality 
in (\ref{cnd}), combined with (\ref{dfe}), implies (\ref{dfm}), and hence 
(\ref{xdd}.iii). The third equality of (\ref{cnd}) is in turn obvious from 
(\ref{dfe}) for $\,\jt=0$.

We have thus proved (\ref{cnd}) and (\ref{dfe}), as well as the fact that they 
imply (\ref{xdd}.iii). They similarly yield the remaining equations in 
(\ref{xdd}) -- (\ref{int}). More precisely, (\ref{xdd}.ii) is nothing else 
than (\ref{dfm}) (that is, (\ref{dfe})) multiplied by 
$\,\dot\jy=\ddot\jy=\dot\jy_0^{\phantom j}e^\jt\nh\ne0$, (\ref{xdd}.i) is an 
obvious consequence of (\ref{xex}) -- (\ref{dfe}), while (\ref{int}) can 
easily be verified directly.
\end{proof}

\section{The Koi\-\hbox{so\hskip.7pt-}\hskip0ptCao examples}\label{pk}
We fix $\,m,k\in\bbZ\,$ with $\,m>k>0\,$ and define $\,S:(0,\infty)\to\bbR\,$ 
by
\begin{equation}\label{sae}
S(a)\,\,=\,\displaystyle{\int_0^{\hskip3.9ptQ}}[\kp^{m-1}+(k-m)\kp^m]\hh e^{-\hh a\kp}\hs d\kp\hh,
\hskip9pt\text{\rm where}\hskip7ptQ\,\,=\,\,\displaystyle{\frac{m+k}{m-k}}\,.
\end{equation}
One then has -- as shown below -- the inequalities
\begin{equation}\label{nrw}
S(0)\,<\,0\,<\,S(a)\hskip9pt\text{\rm whenever}\hskip7pta\,\ge\,m(m-k)\hh.
\end{equation}
We may thus choose $\,a\in(0,m(m-k))\,$ such that $\,S(a)=0$. Next, we 
introduce data (\ref{ini}) satisfying (\ref{cnd}) by setting
$\,\jy_0^{\phantom j}\nh=(1-m/k)a/(m-1)$, 
$\,\jx_0^{\phantom j}\nh=\jy_0^{\phantom j}\nh-\log\hh(m-k)$, 
$\,\fy{}_0^{\phantom j}\nh=0$, $\,\dot\jy_0^{\phantom j}\nh=a/(m-1)$, 
$\,\dot\jx_0^{\phantom j}\nh=\dot\jy_0^{\phantom j}\nh-1$, and 
\hbox{\hskip0pt$\dt{}_0^{\phantom j}\nh=k$}. The solution 
$\,(\jx,\jy,\fy):\bbR\to\rtr$ of (\ref{xdd}) -- (\ref{int}) with a 
constant function $\,\cf$, corresponding to these data as in 
Theorem~\ref{rusjs}, then also satisfies the boundary conditions (\ref{bdr}) 
with $\,\tc=\log\hh[(m+k)/(m-k)]$. In fact -- first, our choice of 
$\,\fy{}_0^{\phantom j}\nh$ and $\,\dt{}_0^{\phantom j}$ gives $\,\fy(0)=0\,$ 
and $\,\dt(0)=k$. Second, $\,\fy(\tc)=0\,$ since, due to 
Theorem~\ref{rusjs}, 
$\,G(\tc)\hs\fy(\tc)=\textstyle{\int_0^{\tc}}F(\jt)\,d\jt\,$ while, evaluating 
the last integral in terms of the new variable $\,\kp=e^\jt$ and using the 
equality $\,\dot\jy_0^{\phantom j}\nh=a/(m-1)\,$ along with (\ref{sae}) and 
(\ref{gte}), as well as our choice of $\,a$, we obtain 
$\,\textstyle{\int_0^{\tc}}F(\jt)\,d\jt=S(a)=0$. Third, the equalities 
$\,(G\hh\fy)\dot{\,}\hs=\hs F\,$ (in Theorem~\ref{rusjs}), $\,\fy(\tc)=0\,$ 
and (\ref{gte}) imply that $\,\dt(\tc)=F(\tc)/G(\tc)
=m-e^{\jy_0^{\phantom j}\nh-\jx_0^{\phantom j}\nh+\hs\tc}\nnh$, and so 
$\,\dt(\tc)=-\hh k$, since 
$\,e^{\jy_0^{\phantom j}\nh-\jx_0^{\phantom j}}\nnh=m-k\,$ and 
$\,e^{\tc}\nh=(m+k)/(m-k)$. Fourth, the derivative of $\,G\hh\fy$, that is, 
the function $\,F\,$ given by (\ref{gte}), has only one zero in $\,\bbR\,$ 
and, consequently, $\,\fy\,$ may have at most two zeros; the relations 
$\,\fy(0)=\fy(\tc)=0<\dt(0)$ thus imply the inequality $\,\fy>0\,$ on 
$\,(0,\tc)$.

Theorem~\ref{ckcrs} states in turn that the above solution 
$\,(\jx,\jy,\fy)\,$ leads to a construction of a K\"ah\-ler-Ric\-ci 
sol\-i\-ton on each of the compact complex manifolds $\,M^m_k$. These are 
the {\it Koi\-\hbox{so\hskip.7pt-}\hskip0ptCao examples}.

\vskip4pt
\begin{proof}[Proof of the inequality\/ {\rm(\ref{nrw})}]For 
$\,m,k,a,Q\in\bbR\,$ such that $\,m\ge1$ and $\,Q\in(0,\infty)$, defining 
$\,S(a)\,$ by (\ref{sae}) (even without assuming that $\,Q=(m+k)/(m-k)$), we 
have $\,S(a)=H_{m-1}^{\phantom i}+(k-m)H_m^{\phantom i}$, where 
$\,H_m^{\phantom i}=\int_0^Q\kp^me^{-\hh a\kp}\hs d\kp\,$ for $\,m\ge0$. 
If $\,m\ge1$, integration by parts yields 
$\,mH_{m-1}^{\phantom i}=aH_m^{\phantom i}+Q^me^{-\hh aQ}\nnh$, and so 
$\,mS(a)=[\hs a-m(m-k)]\hh H_m^{\phantom i}+Q^me^{-\hh aQ}\nnh$. Positivity of 
$\,H_m^{\phantom i}$ thus gives $\,S(a)>0\,$ for $\,a\ge m(m-k)$. However, 
$\,H_m^{\phantom i}=Q^{m+1}/(m+1)\,$ if $\,a=0\,$ and $\,m\ge1$. Our 
formula for $\,mS(a)$ now states that 
$\,m(m+1)Q^{-m}S(0)=-\hh m(m-k)\hh Q+m+1$. With $\,Q=(m+k)/(m-k)\,$ and 
$\,m>k\ge1$, it follows that $\,S(0)<0$.
\end{proof}

\section{Case (iii) in Lemma~\ref{snses}}\label{pt}
The assumptions of Lemma~\ref{rhxpl} will still hold if the quintuple 
$\,m,k,\tc,\jx,\fy$ is replaced by $\,m,k,\tc,\hat\jx,\hat\fy$, where 
$\,\hat\jx(\jt)=\jx(\tc\nnh-\jt)\,$ and $\,\hat\fy(\jt)=\fy(\tc\nnh-\jt)$. 
Performing the construction described in the proof of Lemma~\ref{rhxpl} for 
$\,m,k,\tc,\hat\jx,\hat\fy$, with the same objects $\,N\nh,h,\ey\nh,M\nh$, the 
same fibre norm, and the same Her\-mit\-i\-an connection in $\,\ey\hs$ as 
before, we obtain a new metric $\,\hatg\,$ and function $\,\hat\jt:M\to\bbR$, 
for which we must use a symbol other than $\,\jt$, since it differs in general 
from the function $\,\jt:M\to\bbR$ associated with the original data 
$\,m,k,\tc,\jx,\fy$. To describe how the pair $\,(g,\jt)\,$ is related to 
$\,(\hatg,\hat\jt)$, we use the dif\-feo\-mor\-phism $\,Z:M\to M\,$ which, 
restricted to $\,\ey\smallsetminus N\nh$, acts in every fibre as the standard 
inversion, that is, the division of any nonzero vector by the square of its 
norm. Letting $\,Z^*\nh\hatg\,$ and $\,Z^*\nh\chi\,$ denote the pull\-back 
of the metric $\,\hatg\,$ under $\,Z\,$ and the composition $\,\chi\circ Z$, 
for any function $\,\chi\,$ on $\,\ey\smallsetminus N\nh$, we then have
\begin{equation}\label{zar}
\mathrm{a)}\hskip8ptZ^*\nh r\,=\,1/r\hs,\hskip15pt
\mathrm{b)}\hskip8ptZ^*\nh\hat\jt=\tc\nnh-\jt\hs,\hskip15pt
\mathrm{c)}\hskip8ptZ^*\nh\hatg=e^{\tc-2\jt}g\hs,
\end{equation}
where, as before, $\,r:\ey\smallsetminus N\to\bbR\,$ is the fibre norm.

In fact, (\ref{zar}.a) is a trivial consequence of the definition of $\,Z$. To 
justify (\ref{zar}.b) -- (\ref{zar}.c), note that (\ref{ddr}) remains valid 
after one has replaced $\,\fy(\jt)\,$ by 
$\,\hat\fy(\jt)=\fy(\tc\hskip-1.6pt-\hh\jt)\,$ and $\,\jt(r)\,$ by 
$\,\hat\jt(r)\hs=\hs\tc\hskip-1.6pt-\hs\jt(1/r)$. This last choice of 
$\,\hat\jt:M\to\bbR\,$ easily yields (\ref{zar}.b). The metric  $\,\hatg\,$ 
arises, in turn, from a modified version of the formulae for $\,g\,$ in the 
proof of Lemma~\ref{rhxpl}; the modification amounts to using 
$\,e^{\hat\jt(r)}$ and $\,\hat\fy(\hat\jt(r))=\fy(\jt(1/r))\,$ instead of 
$\,e^\jt$ and $\,\fy\,$ (that is, instead of \ $\,e^{\jt(r)}$ and 
$\,\fy(\jt(r))$), which gives (\ref{zar}.c).

Relation (\ref{zar}.c) states that $\,Z\,$ is an isometry between the 
Riemannian manifolds $\,(M,e^{\tc-2\jt}g)\,$ and $\,(M,\hatg)$. Thus, 
$\,\hatg\,$ is isometric to a metric resulting from a specific con\-for\-mal 
change of $\,g$.

If, in addition, the original data $\,m,k,\tc,\jx,\fy\,$ arise from a 
solution $\,(\jx,\jy,\fy)\,$ of (\ref{xdd}) -- (\ref{bdr}), the function 
$\,\cf:M\to\bbR\,$ is defined by (\ref{sgf}.i), and $\,\hat\cf\,$ denotes its 
analog for the new solution $\,(\hat\jx,\hat\jy,\hat\fy)\,$ obtained by 
substituting$\,\tc\nnh-\jt\,$ for the variable $\,\jt$, then
\begin{equation}\label{zas}
\mathrm{i)}\hskip8ptZ^*\nh\hat\cf\,=\,e^{-\hh\jt\hs+\tc/2}\nh\cf\hs,
\hskip25pt
\mathrm{ii)}\hskip8ptZ^*\nh(\hatg/\hat\cf^2)\,=\,g/\cf^2.
\end{equation}
The first equality is obvious here due to (\ref{sgf}.i) and (\ref{zar}.b), the 
second -- in view of the first one combined with (\ref{zar}.c) and 
multiplicativity of the operation $\,Z^*$ with respect to products of 
functions and tensor fields.

By (\ref{zas}.ii), applying Theorem~\ref{ckcrs} to these solutions 
$\,(\jx,\jy,\fy)\,$ and $\,(\hat\jx,\hat\jy,\hat\fy)\,$ results in {\it 
compact Ric\-ci sol\-i\-tons\/ $\,(M,g/\cf^2)\,$ and\/ 
$\,(M,\hatg/\hat\cf^2)$, which are isometric to each other}.

If, furthermore, $\,(\jx,\jy,\fy)\,$ represents case (iii) in 
Lemma~\ref{snses} then, according to the final clause of Lemma~\ref{snses}, 
the solution $\,(\hat\jx,\hat\jy,\hat\fy)\,$ is an example of case (ii), and 
so the manifold $\,(M,\hatg/\hat\cf^2)\,$ must be isometric to one of the 
Koi\-\hbox{so\hskip.7pt-}\hskip0ptCao examples. The same therefore holds for 
$\,(M,g/\cf^2)$. We have in this way obtained a proof of Maschler's result  
\cite{maschler}.

Since the  standard inversion of the plane is not hol\-o\-mor\-phic (while 
being an\-ti\-hol\-o\-mor\-phic), $\,Z\,$ transforms the original complex 
structure onto a different one, bi\-hol\-o\-mor\-phic to it.

\begin{bibliography}[References]

\bib{aubin}{article}{
   author={Aubin, Thierry},
   title={\'Equations du type Monge-Amp\`ere sur les vari\'et\'es
   k\"ahl\'eriennes compactes},
   language={French, with English summary},
   journal={Bull. Sci. Math. (2)},
   volume={102},
   date={1978},
   number={1},
   pages={63--95},
   issn={0007-4497},
   review={\MR{494932 (81d:53047)}},
}

\bib{bando-mabuchi}{article}{
   author={Bando, Shigetoshi},
   author={Mabuchi, Toshiki},
   title={Uniqueness of Einstein K\"ahler metrics modulo connected group
   actions},
   conference={
      title={Algebraic geometry, Sendai, 1985},
   },
   book={
      series={Adv. Stud. Pure Math.},
      volume={10},
      publisher={North-Holland},
      place={Amsterdam},
   },
   date={1987},
   pages={11--40},
   review={\MR{946233 (89c:53029)}},
}

\bib{barros-ribeiro}{article}{
   author={Barros, A.},
   author={Ribeiro, E.},
   title={Some characterizations for compact almost Ric\-ci sol\-i\-tons},
   journal={Proc. Amer. Math. Soc. (electronic)},
   date={July 22, 2011},
   issn={1088-6826},
   doi={10.1090/S0002-9939-2011-11029-3},
}

\bib{berard-bergery}{article}{
   author={B{\'e}rard-Bergery, Lionel},
   title={Sur de nouvelles vari\'et\'es riemanniennes d'Ein\-stein},
   language={French},
   conference={
      title={Institut \'Elie Cartan, 6},
   },
   book={
      series={Inst. \'Elie Cartan},
      volume={6},
      publisher={Univ. Nancy},
      place={Nancy},
   },
   date={1982},
   pages={1--60},
   review={\MR{727843 (85b:53048)}},
}

\bib{besse}{book}{
   author={Besse, Arthur L.},
   title={Ein\-stein manifolds},
   series={Ergebnisse der Mathematik und ihrer Grenzgebiete (3) [Results in
   Mathematics and Related Areas (3)]},
   volume={10},
   publisher={Springer-Verlag},
   place={Berlin},
   date={1987},
   pages={xii+510},
   isbn={3-540-15279-2},
   review={\MR{867684 (88f:53087)}},
}
  \bib{bourguignon}{article}{
    title={L'espace des m\'etriques riemanniennes d'une vari\'et\'e compacte},
    author={J.\hskip1.9ptP.\hskip1.9ptBourguignon},
    journal={Th\`ese d'Etat, Universit\'e Paris VII},
    date={1974},
  }

\bib{calabi}{article}{
   author={Calabi, Eugenio},
   title={On K\"ahler manifolds with vanishing canonical class},
   conference={
      title={Algebraic geometry and topology. A symposium in honor of S.
      Lefschetz},
   },
   book={
      publisher={Princeton University Press},
      place={Princeton, NJ},
   },
   date={1957},
   pages={78--89},
   review={\MR{0085583 (19,62b)}},
}

\bib{cao-eg}{article}{
   author={Cao, Huai-Dong},
   title={Existence of gradient K\"ah\-ler-Ric\-ci sol\-i\-tons},
   conference={
      title={Elliptic and parabolic methods in geometry (Minneapolis, MN,
      1994)},
   },
   book={
      publisher={A.\ K.\ Peters},
      place={Wellesley, MA},
   },
   date={1996},
   pages={1--16},
   review={\MR{1417944 (98a:53058)}},
  }

\bib{cao-rp}{article}{
   author={Cao, Huai-Dong},
   title={Recent progress on Ric\-ci sol\-i\-tons},
   conference={
      title={Recent advances in geometric analysis},
   },
   book={
      series={Adv. Lect. Math. (ALM)},
      volume={11},
      publisher={Int. Press, Somerville, MA},
   },
   date={2010},
   pages={1--38},
   review={\MR{2648937 (2011d:53061)}},
  }

\bib{chow}{article}{
   author={Chow, Bennett},
   title={Ric\-ci flow and Ein\-stein metrics in low dimensions},
   conference={
      title={Surveys in differential geometry: essays on Ein\-stein manifolds},
   },
   book={
      series={Surv. Differ. Geom., VI},
      publisher={Int. Press, Boston, MA},
   },
   date={1999},
   pages={187--220},
   review={\MR{1798610 (2001m:53116)}},
}

\bib{dancer-wang}{article}{
   author={Dancer, A.},
   author={Wang, M.},
   title={On Ric\-ci sol\-i\-tons of cohomogeneity one},
   journal={Ann. Global Anal. Geom.},
   volume={39},
   date={2011},
   number={3},
   pages={259--292},
   issn={0232-704X},
   review={},
   doi={10.1007/s10455-010-9233-1},
}

\bib{derdzinski-maschler-lc}{article}{
   author={Derdzinski, A.},
   author={Maschler, G.},
   title={Local classification of conformally-Ein\-stein K\"ah\-ler metrics in
   higher dimensions},
   journal={Proc. London Math. Soc. (3)},
   volume={87},
   date={2003},
   number={3},
   pages={779--819},
   issn={0024-6115},
   review={\MR{2005883 (2004i:53051)}},
   doi={10.1112/S0024611503014175},
}

\bib{derdzinski-maschler-sk}{article}{
   author={Derdzinski, A.},
   author={Maschler, G.},
   title={Special K\"ah\-ler-Ric\-ci potentials on compact K\"ah\-ler manifolds},
   journal={J. Reine Angew. Math.},
   volume={593},
   date={2006},
   pages={73--116},
   issn={0075-4102},
   review={\MR{2227140 (2007b:53150)}},
   doi={10.1515/CRELLE.2006.030},
}

\bib{dillen-verstraelen}{collection}{
   title={Handbook of differential geometry. Vol. I},
   editor={Dillen, Franki J. E.},
   editor={Verstraelen, Leopold C. A.},
   publisher={North-Holland},
   place={Amsterdam},
   date={2000},
   pages={xii+1054},
   isbn={0-444-82240-2},
   review={\MR{1736851 (2000h:53003)}},
}

\bib{federbush}{article}{
   author={Federbush, P.},
   title={Partially alternate derivation of a result of Nelson},
   journal={J. Mathematical Phys.},
   volume={10},
   date={1969},
   pages={50--52},
   doi={10.1063/1.1664760},
}

\bib{feldman-ilmanen-knopf}{article}{
   author={Feldman, Mikhail},
   author={Ilmanen, Tom},
   author={Knopf, Dan},
   title={Rotationally symmetric shrinking and expanding gradient
   K\"ah\-ler-Ric\-ci sol\-i\-tons},
   journal={J. Differential Geom.},
   volume={65},
   date={2003},
   number={2},
   pages={169--209},
   issn={0022-040X},
   review={\MR{2058261 (2005e:53102)}},
}

\bib{fernandez-lopez-garcia-rio}{article}{
   author={Fern{\'a}ndez-L{\'o}pez, M.},
   author={Garc{\'{\i}}a-R{\'{\i}}o, E.},
   title={A remark on compact Ric\-ci sol\-i\-tons},
   journal={Math. Ann.},
   volume={340},
   date={2008},
   number={4},
   pages={893--896},
   issn={0025-5831},
   review={\MR{2372742 (2008j:53077)}},
   doi={10.1007/s00208-007-0173-4},
}

\bib{friedan}{article}{
   author={Friedan, Daniel Harry},
   title={Nonlinear models in $2+\varepsilon$ dimensions},
   journal={Ann. Physics},
   volume={163},
   date={1985},
   number={2},
   pages={318--419},
   issn={0003-4916},
   review={\MR{811072 (87f:81130)}},
  }

\bib{gross}{article}{
   author={Gross, Leonard},
   title={Logarithmic Sobolev inequalities},
   journal={Amer. J. Math.},
   volume={97},
   date={1975},
   number={4},
   pages={1061--1083},
   issn={0002-9327},
   review={\MR{0420249 (54 \#8263)}},
}

\bib{hamilton}{article}{
   author={Hamilton, Richard S.},
   title={Three-manifolds with positive Ric\-ci curvature},
   journal={J. Differential Geom.},
   volume={17},
   date={1982},
   number={2},
   pages={255--306},
   issn={0022-040X},
   review={\MR{664497 (84a:53050)}},
  }

\bib{hamilton-tr}{article}{
   author={Hamilton, Richard S.},
   title={The Ric\-ci flow on surfaces},
   conference={
      title={Mathematics and general relativity},
      address={Santa Cruz, CA},
      date={1986},
   },
   book={
      series={Contemp. Math.},
      volume={71},
      publisher={Amer. Math. Soc.},
      place={Providence, RI},
   },
   date={1988},
   pages={237--262},
   review={\MR{954419 (89i:53029)}},
  }

\bib{ivey}{article}{
   author={Ivey, Thomas},
   title={Ric\-ci sol\-i\-tons on compact three-manifolds},
   journal={Differential Geom. Appl.},
   volume={3},
   date={1993},
   number={4},
   pages={301--307},
   issn={0926-2245},
   review={\MR{1249376 (94j:53048)}},
   doi={10.1016/0926-2245(93)90008-O},
  }

\bib{kobayashi-ochiai}{article}{
   author={Kobayashi, Shoshichi},
   author={Ochiai, suchushiro},
   title={Characterizations of complex projective spaces and hyperquadrics},
   journal={J. Math. Kyoto Univ.},
   volume={13},
   date={1973},
   pages={31--47},
   issn={0023-608X},
   review={\MR{0316745 (47 \#5293)}},
}

\bib{koiso}{article}{
   author={Koiso, Norihito},
   title={On rotationally symmetric Hamilton's equation for
   K\"ah\-ler-Ein\-stein metrics},
   conference={
      title={Recent topics in differential and analytic geometry},
   },
   book={
      series={Adv. Stud. Pure Math.},
      volume={18},
      publisher={Academic Press},
      place={Boston, MA},
   },
   date={1990},
   pages={327--337},
   review={\MR{1145263 (93d:53057)}},
  }

\bib{li-c}{article}{
   author={Li, Chi},
   title={On rotationally symmetric K\"ah\-ler-Ric\-ci sol\-i\-tons},
   journal={preprint arXiv:1004.4049},
   date={2010},
}

\bib{li-xm}{article}{
   author={Li, Xue-Mei},
   title={On extensions of Myers' theorem}, 
   journal={Bull.\ London Math.\ Soc.\ \textbf{27} (1995), 392--396},
}

\bib{maschler}{article}{
   author={Maschler, Gideon},
   title={Special K\"ah\-ler-Ric\-ci potentials and Ric\-ci sol\-i\-tons},
   journal={Ann. Global Anal. Geom.},
   volume={34},
   date={2008},
   number={4},
   pages={367--380},
   issn={0232-704X},
   review={\MR{2447905 (2009h:53169)}},
   doi={10.1007/s10455-008-9114-z},
}
\bib{page}{article}{
   author={Page, Don N.},
   title={A compact rotating gravitational instanton},
   journal={Phys. Lett. B},
   volume={79},
   date={1978},
   number={3},
   pages={235--238},
   issn={0370-2693},
   doi={10.1016/0370-2693(78)90231-9},
}

\bib{perelman}{article}{
    title={The entropy formula for the Ric\-ci flow and its geometric 
applications},
    author={G.\hskip1.9ptPerelman},
    journal={preprint, arXiv:math.DG/0211159},
}

\bib{perelman-f}{article}{
    title={Ric\-ci flow with surgery on three-man\-i\-folds},
    author={G.\hskip1.9ptPerelman},
    journal={preprint, arXiv:math.DG/0303109}
}

\bib{perelman-r}{article}{
    title={
Finite extinction time for the solutions to the Ric\-ci flow on certain 
three-man\-i\-folds},
    author={G.\hskip1.9ptPerelman},
    journal={preprint, arXiv:math.DG/0307245},
}

\bib{pigola-rigoli-rimoldi-setti}{article}{
   author={Pigola, Stefano},
   author={Rigoli, Marco},
   author={Rimoldi, Michele},
   author={Setti, Alberto G.},
   title={Ric\-ci almost sol\-i\-tons},
   journal={Ann. Scuola Norm. Sup. Pisa Cl. Sci. (4)},
   volume={10},
   date={2011},
   number={4},
   pages={757--799},
   doi={10.2422/2036-2145.2011.4.01},
}

\bib{rothaus}{article}{
   author={Rothaus, Oscar S.},
   title={Logarithmic Sobolev inequalities and the spectrum of Schr\"odinger
   operators},
   journal={J. Funct. Anal.},
   volume={42},
   date={1981},
   number={1},
   pages={110--120},
   issn={0022-1236},
   review={\MR{620582 (83f:58080b)}},
   doi={10.1016/0022-1236(81)90050-1},
  }

\bib{tian-pc}{article}{
    author={G.\hskip1.9ptTian},
    journal={private communication},
    date={2005},
}

\bib{tian-zhu}{article}{
   author={Tian, Gang},
   author={Zhu, Xiaohua},
   title={A new holomorphic invariant and uniqueness of 
   K\"ah\-ler-Ric\-ci sol\-i\-tons},
   journal={Comment. Math. Helv.},
   volume={77},
   date={2002},
   number={2},
   pages={297--325},
   issn={0010-2571},
   review={\MR{1915043 (2003i:32042)}},
   doi={10.1007/s00014-002-8341-3},
}

\bib{wang-zhu}{article}{
   author={Wang, Xu-Jia},
   author={Zhu, Xiaohua},
   title={K\"ah\-ler-Ric\-ci sol\-i\-tons on toric manifolds with positive 
   first Chern class},
   journal={Adv. Math.},
   volume={188},
   date={2004},
   number={1},
   pages={87--103},
   issn={0001-8708},
   review={\MR{2084775 (2005d:53074)}},
   doi={10.1016/j.aim.2003.09.009},
}

\bib{yau}{article}{
   author={Yau, Shing Tung},
   title={On the Ricci curvature of a compact K\"ahler manifold and the
   complex Monge-Amp\`ere equation. I},
   journal={Comm. Pure Appl. Math.},
   volume={31},
   date={1978},
   number={3},
   pages={339--411},
   issn={0010-3640},
   review={\MR{480350 (81d:53045)}},
   doi={10.1002/cpa.3160310304},
}

\bib{zhang}{article}{
   author={Zhang, Zhenlei},
   title={On the finiteness of the fundamental group of a compact shrinking
   Ric\-ci sol\-i\-ton},
   journal={Colloq. Math.},
   volume={107},
   date={2007},
   number={2},
   pages={297--299},
   issn={0010-1354},
   review={\MR{2284167 (2007h:53067)}},
   doi={10.4064/cm107-2-9},
}

\end{bibliography}
\end{document}